\documentclass[a4,12pt]{article}
\usepackage{graphicx,fancyhdr,mathrsfs,bbm}
\usepackage{amsfonts}
\usepackage{amsmath}
\usepackage{mathtools}
\usepackage{multicol}
\usepackage{mathrsfs}
\usepackage{multirow}
\usepackage{amssymb}
\usepackage{url}
\usepackage[colorlinks=true,citecolor=blue,urlcolor=blue]{hyperref}
\usepackage{fancyhdr}
\usepackage{indentfirst}
\usepackage{amsthm}
\usepackage{comment}
\usepackage{dsfont}
\usepackage{natbib}
\usepackage[misc]{ifsym}
\usepackage[toc,page]{appendix}
\usepackage{subfig}
\usepackage[top=20mm, bottom=30mm, left=20mm, right=20mm]{geometry}
\usepackage{enumitem}
\setlist{
  align=left,
  labelindent=0mm,
  leftmargin=!,
  itemindent=0mm,
  listparindent=\parindent,
  parsep=0mm,
  topsep=1mm,
  itemsep=1mm
}
\setlist[itemize,1]{label={\mysquare}, labelwidth=\widthof{\mysquare\ }}
\setlist[itemize,2]{label={\mytriangle}, labelwidth=\widthof{\mytriangle\ }}
\setlist[itemize,3]{label={\mybar}, labelwidth=\widthof{\mybar\ }}
\setlist[itemize,4]{label={\mydot}, labelwidth=\widthof{\mydot\ }}
\setlist[enumerate,1]{label=(\arabic*), labelwidth=\widthof{(9)}}
\setlist[enumerate,2]{label=(\arabic{enumi}.\arabic*), labelwidth=\widthof{(9.9)}}
\setlist[enumerate,3]{label=(\arabic{enumi}.\arabic{enumii}.\arabic*), labelwidth=\widthof{(9.9.9)}}
\setlist[enumerate,4]{label=(\arabic{enumi}.\arabic{enumii}.\arabic{enumiii}.\arabic*), labelwidth=\widthof{(9.9.9.9)}}
\usepackage{actuarialsymbol}
\usepackage{actuarialangle}

\newcommand{\E}{\mathbb{E}}
\newcommand{\R}{\mathbb{R}}

\newcommand{\N}{\mathbb{N}}

\newcommand{\id}{\mathds{1}}

\newcommand{\bx}{\bm{x}}

\newcommand{\bX}{\bm{X}}

\newcommand{\bu}{\bm{u}}

\newcommand{\bzero}{\bm{0}}
\newcommand{\bone}{\bm{1}}

\renewcommand{\ge}{\geqslant}
\renewcommand{\le}{\leqslant}
\renewcommand{\geq}{\geqslant}
\renewcommand{\leq}{\leqslant}
\renewcommand{\epsilon}{\varepsilon}

\usepackage{array}
\newcommand{\PreserveBackslash}[1]{\let\temp=\\#1\let\\=\temp}
\newcolumntype{C}[1]{>{\PreserveBackslash\centering}p{#1}}
\newcolumntype{R}[1]{>{\PreserveBackslash\raggedleft}p{#1}}
\newcolumntype{L}[1]{>{\PreserveBackslash\raggedright}p{#1}}

\usepackage{csquotes}
\renewcommand{\mkbegdispquote}[2]{\itshape}

\usepackage{booktabs,array}

\newcount\rowc

\makeatletter
\def\ttabular{%
  \hbox\bgroup
  \let\\\cr
  \def\rulea{\ifnum\rowc=\@ne \hrule height 1.3pt \fi}
  \def\ruleb{
    \ifnum\rowc=1\hrule height 1.3pt \else
      \ifnum\rowc=6\hrule height \heavyrulewidth
      \else \hrule height \lightrulewidth\fi\fi}
  \valign\bgroup
  \global\rowc\@ne
  \rulea
  \hbox to 10em{\strut \hfill##\hfill}%
  \ruleb
  &&%
  \global\advance\rowc\@ne
  \hbox to 10em{\strut\hfill##\hfill}%
  \ruleb
  \cr}
\def\endttabular{%
  \crcr\egroup\egroup}

\theoremstyle{definition}
\newtheorem{Theorem}{Theorem}
\newtheorem{Corollary}{Corollary}

\newtheorem{Proposition}{Proposition}
\newtheorem{Definition}{Definition}
\newtheorem{Example}{Example}

\newtheorem{Remark}{Remark}

\usepackage{tikz}

\usepackage[compact]{titlesec}

\usepackage{color,xcolor}
\usepackage{bm}
\newcommand{\rev}[1]{{\color{black}#1}}

\newcommand{\two}{\mathrm{I}\hspace{-1.2pt}\mathrm{I}}
\newcommand{\one}{\mathrm{I}}

\newcommand*{\U}{\operatorname{U}}

\renewcommand*{\P}{\mathbb{P}}

\newcommand*{\ES}{\operatorname{ES}}

\newcommand*{\psii}{{\psi^{-1}}}

\usepackage{algpseudocode}
\usepackage{algorithm,algcompatible} 

\algnewcommand\algorithmicinput{\textbf{Input:}}
\algnewcommand\INPUT{\item[\algorithmicinput]}

\algnewcommand\algorithmicoutput{\textbf{Output:}}
\algnewcommand\OUTPUT{\item[\algorithmicoutput]}

\makeatletter
\DeclareRobustCommand{\bsquare}{%
  \mathop{\vphantom{\sum}\mathpalette\bigstar@\relax}\slimits@
}

\newcommand{\bigstar@}[2]{%
  \vcenter{%
    \sbox\z@{$#1\sum$}%
    \hbox{\resizebox{.9\dimexpr\ht\z@+\dp\z@}{!}{$\m@th\dsquare$}}%
  }%
}
\makeatother

\usepackage[onehalfspacing]{setspace}

\begin{document}

\title{
Measuring multivariate maximal tail dependence
}
\author{
  Takaaki Koike\thanks{\protect\linespread{1}\protect\selectfont Corresponding author.}\,\,\thanks{\protect\linespread{1}\protect\selectfont
    Graduate School of Economics, Hitotsubashi University, 2-1, Naka, Kunitachi, Tokyo 186-8601, Japan.
    Email: \texttt{takaaki.koike@r.hit-u.ac.jp}},
  Marius Hofert\thanks{\protect\linespread{1}\protect\selectfont
    Department of Statistics and Actuarial Science, The University of Hong Kong, Pok Fu Lam, Hong Kong.
     Email: \texttt{mhofert@hku.hk}}\, and
  Haruki Tsunekawa\thanks{\protect\linespread{1}\protect\selectfont
    Graduate School of Economics, Hitotsubashi University, 2-1, Naka, Kunitachi, Tokyo 186-8601, Japan.
    Email: \texttt{em265009@g.hit-u.ac.jp}}
}

\maketitle

\begin{abstract}
  The classical tail dependence coefficient (TDC) may fail to capture
  non-exchangeable features of bivariate tail dependence since it evaluates the
  underlying copula only along the diagonal. To address this limitation, several
  measures of strongest manifestation of tail dependence have been proposed in
  the bivariate case, including a measure based on the tail copula of the
  underlying bivariate copula. This paper introduces and investigates the
  multivariate maximal tail concordance measure (MTCM), which extends the
  bivariate measure to the multivariate case. The MTCM quantifies the largest
  tail mass over lower hyperrectangles of common unit volume, while the
  associated maximizer identifies the direction of maximal tail probability. We
  establish fundamental properties of the MTCM in the multivariate case,
  including existence of an optimal direction.  We also derive analytical
  representations for several important model classes. Closed-form expressions
  are further obtained for survival Marshall–Olkin copulas, Archimax and nested
  Archimedean copulas with regularly varying Archimedean generators.  \rev{An
    application to trivariate annual sea-level maxima in England reveals
    off-diagonal stress directions and differences in extremal dependence that
    are not summarized by likelihood values or the TDC alone. 
    A basis-risk
    analysis of a  parametric insurance further shows that the
    fitted model with relatively low values of TDC and MTCM can nevertheless
    yield the most conservative risk assessment due to its non-exchangeability
    in the tail.}
\end{abstract}


\section{Introduction}
A copula is a distribution function with standard uniform margins and a standard
tool for modeling stochastic dependence, in particular joint extremal dependence
in applications in finance, insurance and risk management;
see~\citet{nelsen2006introduction}, \citet{jaworskidurantehaerdlerychlik2010} or
\citet[Chapter~7]{mcneil2015quantitative}. By considering suitable reflections,
we can, without loss of generality, focus on the lower tail (around the origin)
in this work. In the bivariate case, the classical (lower) \emph{tail dependence
  coefficient (TDC)} of \citet{sibuya1960bivariate} is
\begin{align*}
  \lambda(C)=\lim_{u\downarrow 0}\frac{\mathbb{P}(U_1\le u,U_2\le u)}{u}=\lim_{u\downarrow 0}\frac{C(u,u)}{u},\quad (U_1,U_2)\sim C.
\end{align*}
The TDC is widely used to summarize the strength of bivariate tail
dependence. However, since only the diagonal of $C$ enters the definition of
$\lambda(C)$, this notion may miss non-exchangeable or off-diagonal features of
tail dependence.

Several approaches have been proposed to describe or quantify more detailed tail
behavior, such as the tail copula~\citep{schmidt2006non}, tail dependence
function~\citep{joe2010tail} or the tail order function~\citep{hua2011tail} on
the functional side; for scalar summaries, see
\citet{krupskii2015tail}, \citet{lee2018tail}, \citet{hua2019assessing},~\citet{li2023estimation} and
\citet{siburg2024comparing}. In the bivariate case,~\citet{koike2023measuring} introduced the \emph{maximal tail
  concordance measure} (MTCM)
\begin{align*}
  \lambda^\ast(C)=\sup_{b\in(0,\infty)}\Lambda\biggl(b,\frac{1}{b};C\biggr),
\end{align*}
where
\begin{align*}
  \Lambda(x_1,x_2;C)=\lim_{t\downarrow 0}\frac{C(tx_1,tx_2)}{t}, \quad (x_1,x_2)\in(0,\infty)^2,
\end{align*}
denotes the tail copula of $C$. The MTCM $\lambda^\ast(C)$ is also relevant to
the notion of path-based tail dependence proposed by \citet{furman2015paths};
see~\citet{koike2026tail} for their relationship.

The aim of this paper is to develop an extension of the MTCM to the multivariate
case. The main idea is to maximize the tail copula over all hyperrectangles
anchored at the origin and of unit volume. Hence the MTCM is not only a scalar
measure of extremal dependence, but its maximizer also identifies an off-diagonal
stress direction along which tail dependence is strongest.

This proposal fits into a broader literature on multivariate tail summaries.
Several multivariate extensions of the bivariate TDC summarize extremal
dependence, such as~\citet{frahm2006extremal} and
\citet{schmidschmidt2007b}. Within the extreme value framework, the stable tail
dependence function provides the fundamental asymptotic object, and the extremal
coefficient~\citep{schlather2002inequalities,schlather2003dependence} and
orthant tail dependence parameters~\citep{li2009orthant} provide informative
scalar summaries.  More recently, direction-sensitive approaches have been
developed via multivariate directional tail-weighted dependence
measures~\citep{li2024multivariate} and tail dependence
orderings~\citep{siburg2024multivariate}.  Properties of some multivariate tail
dependence measures have been studied in~\citet{fernandez2016independence}
and~\citet{gijbels2020multivariate}.  The multivariate MTCM complements these
approaches by focusing on maximal tail mass over comparable hyperrectangles.

\rev{ As a concrete analysis of multivariate non-exchangeable tail dependence,
  we consider Archimax and nested Archimedean copulas.  An ordinary multivariate
  Archimedean copula is generated by one common generator and is fully
  exchangeable; in particular, all of its margins of the same dimension are
  equal. Archimax and nested Archimedean copulas relax different aspects of this
  construction. An Archimax
  copula~\citep{caperaa2000bivariate,mesiar2013d,charpentierfougeresgenestneslehova2014}
  replaces the symmetric sum in the Archimedean construction by a stable tail
  dependence function, which can incorporate asymmetry and encompasses
  the class of extreme value copulas. A nested Archimedean
  copula~\citep{joe1993parametric,joe1997multivariate,mcneil2008,hofert2010c,hofert2011a}
  instead assigns different generators to the nodes of a rooted tree with leaves
  denoting copula arguments, thereby allowing blocks of margins to exhibit
  different types or degrees of Archimedeanity, for example within-block
  Archimedean dependence to differ from between-block Archimedean dependence,
  and thus partial exchangeability. These extensions are relevant to actuarial
  science and risk management. Empirical applications of Archimax copulas often
  concern hydrological extremes and financial or insurance risk; see
  \citet{chatelain2020inference} or \citet{fakhfekh2024dependence}.
  Nested Archimedean copulas have more direct actuarial and financial
  applications, including asset portfolios \citep{savu2010hierarchies}, sectoral
  default dependence and CDO pricing \citep{hofertscherer2011}, dependence between
  loss triangles from different insurance lines \citep{abdallah2015modeling},
  and cross-country mortality dependence and mortality-linked securities
  \citep{li2021hierarchical}.  }

The paper is organized as follows. In Section~\ref{sec:multivariate:MTCM}, we
introduce the multivariate notion of MTCM and establish its basic properties. In
particular, we show that the appearing supremum is always attained, so the
measure is well defined as a maximum. Sections~\ref{sec:mtcm:ev}
and~\ref{sec:archimax:hac} derive an analytical representation for various
copula families. In Section~\ref{sec:mtcm:ev}, we first find a representation
for the survival copula $\hat C$ of a copula $C$ in the maximum domain of
attraction of an extreme value copula. This leads to a closed-form expression
for the MTCM and its maximizer for survival Marshall--Olkin (MO) copulas, which
naturally arise from common-shock constructions; see
\citet{marshall1967multivariate} and \citet{lindskog2003common} for applications
and~\citet{li2008tail} for tail dependence properties. In
Section~\ref{sec:archimax:hac}, we then consider Archimax and nested Archimedean
copulas with regularly varying generators. In
Section~\ref{sec:numerical:real:data}, we illustrate the practical usefulness of
the multivariate MTCM in applications. \rev{Section~\ref{sec:model:comparison}
  considers an application to trivariate annual sea-level maxima in England
  studied by \citet{smith1990statistics} and \citet{tawn1990modelling}. We
  compare several extreme value models and reveal differences in their MTCM
  values and maximizing directions that are not summarized by likelihood values
  or the TDC alone. In Section~\ref{sec:basis:risk} we then assess basis risk
  under a parametric insurance and demonstrate that a model with
  relatively low values of the TDC and MTCM can nevertheless yield the most
  conservative risk assessment due to its highly non-exchangeable
  tail structure.}  Section~\ref{sec:concl} ends the main part of the paper with
a conclusion. All proofs are deferred to Appendix~\ref{sec:proofs}.

\section{A measure of multivariate maximal tail dependence}\label{sec:multivariate:MTCM}

In this section, we introduce a multivariate extension of the MTCM and study its basic properties. \rev{Throughout the paper, we assume $d\ge 2$.}

Let $C$ be a $d$-dimensional copula. If the limit
\begin{align*}
  \Lambda(\bx;C):=\lim_{t \downarrow 0} \frac{C(t\bx)}{t},
\end{align*}
exists for every $\bx\in(0,\infty)^d$, then we say that $C$ admits a \emph{(lower) tail copula} $\Lambda:(0,\infty)^d\to[0,\infty)$.
The tail copula $\Lambda$ is \emph{non-degenerate} if
$\Lambda \not \equiv 0$, otherwise it is \emph{degenerate}.
Then our proposed multivariate extension of the MTCM is defined as follows.

\begin{Definition}[Multivariate extension of the MTCM]
  Let $\mathcal B = \{\bm{b}\in(0,\infty)^d:\ \prod_{j=1}^d b_j = 1\}$
  and $C$ be a $d$-dimensional copula with tail copula $\Lambda$.  The (multivariate)
  \emph{maximal tail concordance measure (MTCM)} of $C$ is defined as
  \begin{align}\label{eq:def:mtcm:multivariate}
    \lambda^\ast(C) = \sup_{\bm{b}\in \mathcal B} \Lambda(\bm{b};C).
  \end{align}
  The unique maximizer of $\Lambda(\bm{b};C)$ \rev{on $\mathcal B$},
  if it exists, is denoted by $\bm{b}^\ast=\bm{b}^\ast(C)\in \mathcal B$.
\end{Definition}

\rev{Note that uniqueness of the maximizer does not hold in general as addressed
  in~\citet[Remark~2]{koike2026tail}, but is commonly satisfied for typical
  copulas in the bivariate case; see~\citet[Section 3.3]{koike2023measuring}
  and~\citet[Section~4]{koike2026tail}.  } Let $I_{\bm b}=\prod_{j=1}^d[0,b_j]$,
$\bm{b}\in[0,\infty)^d$, and $\bm{U}\sim C$.  Similar to the bivariate case, the
MTCM $\lambda^\ast(C)$ quantifies the maximal possible tail probability
$\lim_{t \downarrow 0}\mathbb{P}(\bm{U}/t \in I_{\bm{b}})/t$ over all
hyperrectangles $I_{\bm{b}}$ with unit volume and lower endpoint anchored at
$\bm{0}_d=(0,\dots,0)\in\R^d$.  \rev{Therefore, the maximizing direction
  $\bm b^\ast$ represents the levels of the
  asymptotically most likely joint tail event among these equal-volume
  hyperrectangles.  Since $\prod_{j=1}^d b_j^\ast=1$, the case $b_j^\ast>1$ ($b_j^\ast<1$)
indicates a side wider (narrower, resp.) than the common diagonal scale $b_j^\ast=1$. 
This relative interpretation is valid
  independently of the dimension $d$.  
  In different dimensions, the same numerical
  value of $b_j^\ast$ should not be interpreted as the same degree of directional
  imbalance.}

\rev{
\begin{Remark}[Measure of imbalance]
  A dimension-adjusted summary of the imbalance of $\bm b^\ast$ can be obtained from compositional data analysis. Let $\tilde{\bm b}^\ast=\bm b^\ast/\sum_{k=1}^d b_k^\ast$. Then $\tilde{\bm b}^\ast$ is a composition satisfying $\sum_{j=1}^d\tilde b_j^\ast=1$. Let $g(\tilde{\bm b}^\ast)=(\prod_{j=1}^d\tilde b_j^\ast)^{1/d}$ denote the geometric mean of $\tilde{\bm b}^\ast$. Since $\prod_{j=1}^d b_j^\ast=1$, we have $g(\tilde{\bm b}^\ast)=1/\sum_{k=1}^d b_k^\ast$ and hence $\tilde b_j^\ast/g(\tilde{\bm b}^\ast)=b_j^\ast$. Therefore, the \emph{Aitchison-norm-based inequality index} of \citet[Section~3]{egozcue2019compositional} is
  \begin{align*}
   \xi(\bm{b}^\ast):=\frac{1}{d}\sum_{j=1}^d\biggl\{\log\biggl(\frac{\tilde b_j^\ast}{g(\tilde{\bm b}^\ast)}\biggr)\biggr\}^2=\frac{1}{d}\sum_{j=1}^d(\log b_j^\ast)^2=\frac{1}{d^2}\sum_{1\leq j<k\leq d}(\log b_j^\ast -\log b_k^\ast)^2,
  \end{align*}
  where the last equality follows from $\sum_{j=1}^d\log b_j^\ast=0$.
  Note that $\xi(\bm{b}^\ast)$ is
  the sample variance of $(\log b_1^\ast ,\ldots,\log b_d^\ast)$.  It is invariant under
  permutations.
  In addition, $\xi(\bm{b}^\ast)=0$ if and only if $\bm b^\ast=\bone_d$. 
  If the maximizer is
  unique, then $\xi(\bm b^\ast)>0$ rules out tail exchangeability, whereas
  $\xi(\bm b^\ast)=0$ does not imply tail exchangeability.
\end{Remark}
}

The following proposition summarizes basic properties of
MTCMs.
To this end, let $M(\bm{u})=\min\{u_1,\dots,u_d\}$, $\bm{u}\in [0,1]^d$, denote
the comonotone copula.

\begin{Proposition}[Basic properties of the MTCM]\label{prop:basic:properties:mtcm}
  Let $C$, $C_1$, $C_2$ be $d$-dimensional copulas admitting tail copulas $\Lambda$, $\Lambda_1$, $\Lambda_2$, respectively.
  \begin{enumerate}[label=(\roman*), labelwidth=\widthof{(iii)}]
  \item\label{item:attainability:mv} The supremum in~\eqref{eq:def:mtcm:multivariate} is attained; hence $\lambda^\ast(C)=\max_{\bm b\in\mathcal B}\Lambda(\bm b;C)$.
  \item\label{item:one:mv} $\lambda^\ast(C)=1$ if and only if $\Lambda(\cdot;C)= \Lambda(\cdot;M)$.
  \item\label{item:zero:mv} $\lambda^\ast(C)=0$ if and only if $\Lambda(\cdot;C)\equiv 0$.
  \item\label{item:monotonocity:mv} $\lambda^\ast(C_1)\le \lambda^\ast(C_2)$ if $\Lambda_1(\bx;C_1) \le \Lambda_2(\bx;C_2)$ for all $\bx\in(0,\infty)^d$.
  \item\label{item:convexity:mv} $\lambda^\ast(tC_1 + (1-t)C_2)\le t \lambda^\ast(C_1) + (1-t)\lambda^\ast(C_2)$ for all $t \in [0,1]$.
  \item\label{item:continuity:mv} Let $(C_n)_{n \in \N}$ be a sequence of copulas.
    If $\Lambda(\cdot;C_n)\to \Lambda(\cdot;C)$ pointwise, then $\lambda^\ast(C_n)\to \lambda^\ast(C)$.
  \end{enumerate}
\end{Proposition}

Attainability~\ref{item:attainability:mv} follows from the fact that $\Lambda$
is continuous and vanishes at the boundaries. Property~\ref{item:one:mv} states
that $\lambda^\ast(C)=1$ if and only if $C$ is \emph{tail comonotonic};
see~\citet{hua2012tail}, \citet{hua2012tail2} and \citet{cheung2019additivity}
for the notion of tail comonotonicity. Property~\ref{item:zero:mv} shows that
$\lambda^\ast(C)=0$ if and only if $C(t\bx)$ converges to $0$ faster than
$t$. By~\ref{item:monotonocity:mv}, $\lambda^\ast$ is monotone with respect to
the tail dependence order introduced in~\citet{siburg2024multivariate}.
\rev{The convexity property~\ref{item:convexity:mv} may be relevant when a
  mixture copula represents heterogeneous populations or latent regimes.  Since
  the MTCM of the mixture is obtained by a single direction for the pooled
  population, pooling regimes with different extremal directions can mask
  regime-specific tail dependence. Finally,~\ref{item:continuity:mv} means that
  $\lambda^\ast$ is continuous with respect to pointwise convergence of tail
  copulas, justifying stability of copula approximations.}

\rev{Next, we summarize some properties of $\bm b^\ast$.  In the following
  proposition, we write $\lambda(C)= \Lambda(\bm 1_d;C)$.  Moreover, since
  argmax sets are non-empty by Proposition~\ref{prop:basic:properties:mtcm}~\ref{item:attainability:mv}, we
  write
  \begin{align*}
  \mathcal B^\ast(C) = \operatorname{argmax}_{\bm b\in\mathcal B} \Lambda(\bm
  b;C).
  \end{align*}
  We also denote by $\mathfrak S_d$ the set of all permutations of
  $\{1,\dots,d\}$ and $\bm x_\pi=(x_{\pi(1)},\dots,x_{\pi(d)})$ for
  $\pi \in \mathfrak S_d$.  In addition, we write
  \begin{align}\label{eq:K:c}
    \mathcal B_{\ge \eta}
    =
    \{\bm b\in\mathcal B: b_j\geq \eta,\, j=1,\ldots,d\},\qquad \eta\in(0,1].
  \end{align}

Since $b_j=1/\prod_{i\neq j}b_i\le \eta^{-(d-1)}$ for every $j\in\{1,\dots,d\}$, 
  we have $\mathcal B_{\ge \eta}\subseteq [\eta,\eta^{-(d-1)}]^d$. 
    This implies that  $\mathcal B_{\ge \eta}$ is closed in $\R^d$ and bounded, and thus compact.
 
  \begin{Proposition}[Basic properties of $\bm b^\ast$ and $\mathcal B(\eta)$]
    \label{prop:basic:properties:bast}
    Let $C$ be a $d$-dimensional copula with a non-degenerate tail copula.
    \begin{enumerate}[label=(\roman*), labelwidth=\widthof{(iii)}]
    \item\label{item:maximizer:permutation:symmetry}
      Let $\mathcal S\subset\mathfrak S_d$ be non-empty.
      Suppose that $\Lambda(\bm x_\pi;C)=\Lambda(\bm x;C)$ for every $\bm x\in(0,\infty)^d$ and $\pi\in \mathcal S$.
      If $\mathcal B^\ast(C)=\{\bm b^\ast\}$, then $\bm b_\pi^\ast
      =
      \bm b^\ast$ for every $\pi\in \mathcal S$.
    \item\label{item:Aast:compact} $\mathcal B^\ast(C)$ is a compact set with $\mathcal B^\ast(C)
      \subseteq
      \mathcal B_{\ge \lambda^\ast(C)}
      \subseteq
      \mathcal B_{\ge \lambda(C)}$.
    \item\label{item:inequalities} Let $\bm b^\ast\in\mathcal B^\ast(C)$ and $I\subseteq\{1,\ldots,d\}$ be non-empty.
      Then \begin{align*}
      \lambda(C)^{|I|} \leq \lambda^\ast(C)^{|I|} \leq \prod_{j\in I}b_j^\ast \leq \lambda^\ast(C)^{-(d-|I|)}\le  \lambda(C)^{-(d-|I|)}
      \end{align*}
      and 
$|I|\lambda(C)\le |I|\lambda^\ast(C) \le 
      \sum_{j\in I}b_j^\ast$.
    \item\label{item:maximizer:exact:range}
    For $\eta\in (0,1]$, the set $\mathcal B_{\ge \eta}$ coincides with the set of all possible unique directions with
  a given value of MTCM $\eta$.
  Namely,
  \begin{align*}
   \mathcal B_{\ge \eta}=\{\bm b\in \mathcal B: \text{there exists a $d$-dimensional copula $C$ such that $\lambda^\ast(C)=\eta$ and $\mathcal B^\ast(C)=\{\bm b\}$}\}.
  \end{align*}
\item\label{item:maximizer:first:order}
      Suppose that $\Lambda$ is continuously differentiable in a
      neighborhood of $\bm b^\ast\in\mathcal B^\ast(C)$. Then
      \begin{align*}
        b_j^\ast
        \frac{\partial}{\partial x_j}\Lambda(\bm x;C) \Big|_{\bm{x} = \bm{b}^\ast}
        =
        \frac{\lambda^\ast(C)}{d},
        \qquad j=1,\ldots,d.
      \end{align*}
    \item\label{item:continuity:b}
      Let $(C_n)_{n \in \N}$ be a sequence of copulas such that $\Lambda(\cdot;C_n)\to \Lambda(\cdot;C)$ pointwise.
      Suppose that $\mathcal B^\ast(C)=\{\bm b^\ast\}$.
      Then $\bm b_n^\ast\rightarrow \bm b^\ast$ for every sequence $(\bm b_n^\ast)$ with $\bm b_n^\ast\in\mathcal B^\ast(C_n)$.
    \end{enumerate}
\end{Proposition}

\rev{
It follows from~\ref{item:maximizer:permutation:symmetry}  that a unique
maximizing direction inherits the permutation symmetries of the tail
copula. In particular, full permutation
invariance implies $\bm b^\ast=\bm 1_d$.
Properties~\ref{item:Aast:compact} and~\ref{item:inequalities}
localize every maximizing direction in a compact region controlled by
$\lambda^\ast(C)$, with $\lambda(C)$ providing a coarser bound, and
thereby yield a compact domain for numerical optimization.
Property~\ref{item:maximizer:exact:range} shows
that this localization is sharp: for a fixed MTCM $\eta$, every direction
in $\mathcal B_{\ge \eta}$ can occur as the unique maximizer of some copula, and
hence no smaller universal range can be deduced from $\eta$ alone. Under
differentiability, Property~\ref{item:maximizer:first:order} states that
the marginal contributions $b_j^\ast(\partial/\partial x_j)\Lambda(\bm x;C) |_{\bm{x} = \bm{b}^\ast}$
are equal across coordinates, providing a balance condition.
Finally, Property~\ref{item:continuity:b} establishes
the stability of a unique maximizing direction under pointwise convergence of tail copulas.
}}

\section{MTCM for extreme value copulas}\label{sec:mtcm:ev}

\subsection{Copulas in a maximum domain of attraction}\label{sec:copulas:doa}

A $d$-dimensional \emph{extreme value copula (EVC)} is given by
\begin{align}
  C_\ell(\bm{u})=\exp(-\ell(-\log u_1,\dots,-\log u_d)),\quad \bm{u}=(u_1,\dots,u_d) \in [0,1]^d,\label{eq:EVC:form}
\end{align}
where $\ell:[0,\infty)^d\to [0,\infty)$ is a \emph{stable tail dependence
  function}; see~\citet{gudendorf2010extreme},
\citet[Section~3.15]{joe2015dependence} and~\rev{\citet[Section~16.3]{mcneil2015quantitative}} on EVCs, and
\citet{ressel2013} for a characterization of stable tail dependence functions in
terms of D-norms. Stable tail dependence functions $\ell$ are convex,
$1$-homogeneous and satisfy
$\max(x_1,\dots,x_d)\le\ell(\bx)\le \sum_{j=1}^d x_j$, $\bx\in[0,\infty)^d$.  A
$d$-dimensional copula $C$ is said to be in the maximum domain of attraction of
an EVC $C_\ell$ if $\lim_{n \to \infty}C({\bm u}^{1/n})^n = C_{\ell}({\bm u})$
for all $\bm u\in[0,1]^d$, \rev{where $\bm u^{1/n}=(u_1^{1/n},\dots,u_d^{1/n})$}.  Since $C_{\ell}({\bm u}^{1/n})^n = C_{\ell}({\bm u})$
for all $\bm u\in[0,1]^d$ and all integers $n\ge 1$, the EVC $C_\ell$ is in the maximum domain of attraction of itself.

In this section we study MTCMs for the (lower) tail of the survival copula of
$C$, where $C$ is in the maximum domain of attraction of an EVC $C_{\ell}$.
This corresponds to focusing on the upper tail of the copula $C$ under
consideration. Denote by $\hat C$ the survival copula of $C$.

The relationship between the tail copula and the stable tail
dependence function follows from the inclusion-exclusion principle; see~\citet{joe2010tail} and~\citet[][Section~8]{rootzen2018multivariate}:
  \begin{align}\label{eq:tail:copula:stdf}
    \Lambda(\bx;\hat{C})
    =
    \sum_{\emptyset\neq S \subseteq \{1,\dots,d\}} (-1)^{|S|-1} \ell_S(\bx),
    \quad \bx\in(0,\infty)^d,
  \end{align}
  where $\ell_S(\bx)=\lim_{x_j\downarrow 0,\,j\notin S}\ell(\bx)$.
As a consequence, we have
  \begin{align*}
    \lambda^\ast(\hat C)
    =
    \max_{\bm{b}\in \mathcal B}
    \sum_{\emptyset\neq S \subseteq \{1,\dots,d\}} (-1)^{|S|-1} \ell_S(\bm{b}).
  \end{align*}
For $d=2$, we have $\ell(x, 0) = \ell(0, x) = x$, $x> 0$, and thus
$\Lambda(x, y;\hat{C}) = x + y - \ell(x, y).$ A number of parametric families of
EVCs have been proposed in the literature, such as the H\"usler-Reiss (HR),
$t$-EV, asymmetric Gumbel and Galambos copulas; see~\citet{tawn1990modelling},
\citet[Chapter~6]{jaworskidurantehaerdlerychlik2010} and
\citet{joe2015dependence}.

\begin{Example}[Multivariate $t$ copulas]
  Multivariate (Student's) $t$ copulas admit tail copulas which coincide with those of the
  corresponding $t$-EVCs~\citep{demarta2005t,nikoloulopoulos2009extreme}. For $d=2$,
  it is shown in~\citet{koike2026tail} that the attaining
  $b^\ast$ is $1$.  For $d= 3$, we observe by simulation of randomized
  correlation matrices $R=(r_{ij})$ that the attaining $\bm b^\ast$ typically deviates
  from $\bone_3=(1,1,1)$. 
  Random correlation matrices $R$ are generated by independently sampling
the entries above the diagonal from $\U(-1,1)$ and retaining only
positive definite matrices.
  Figure~\ref{fig:bast:t} shows the corresponding realizations $(b_1^\ast,b_2^\ast)$ of
  $\bm b^\ast=(b_1^\ast,b_2^\ast,b_3^\ast)$ for the MTCM of a three-dimensional $t$ copula with
  $\nu=5$ degrees of freedom.
  For comparison, we also provide $(b_1^\ast,b_2^\ast)$ under the restriction that $r_{1,3}=r_{2,3}=0.4$.
  \rev{In this case, the tail copula is invariant under interchanging the
    first two coordinates. Hence, whenever the maximizer is unique,
Proposition~\ref{prop:basic:properties:bast}~\ref{item:maximizer:permutation:symmetry} yields
    $b_1^\ast=b_2^\ast$. }
  \begin{figure}[tbp]
    \centering
    \includegraphics[width=\textwidth]{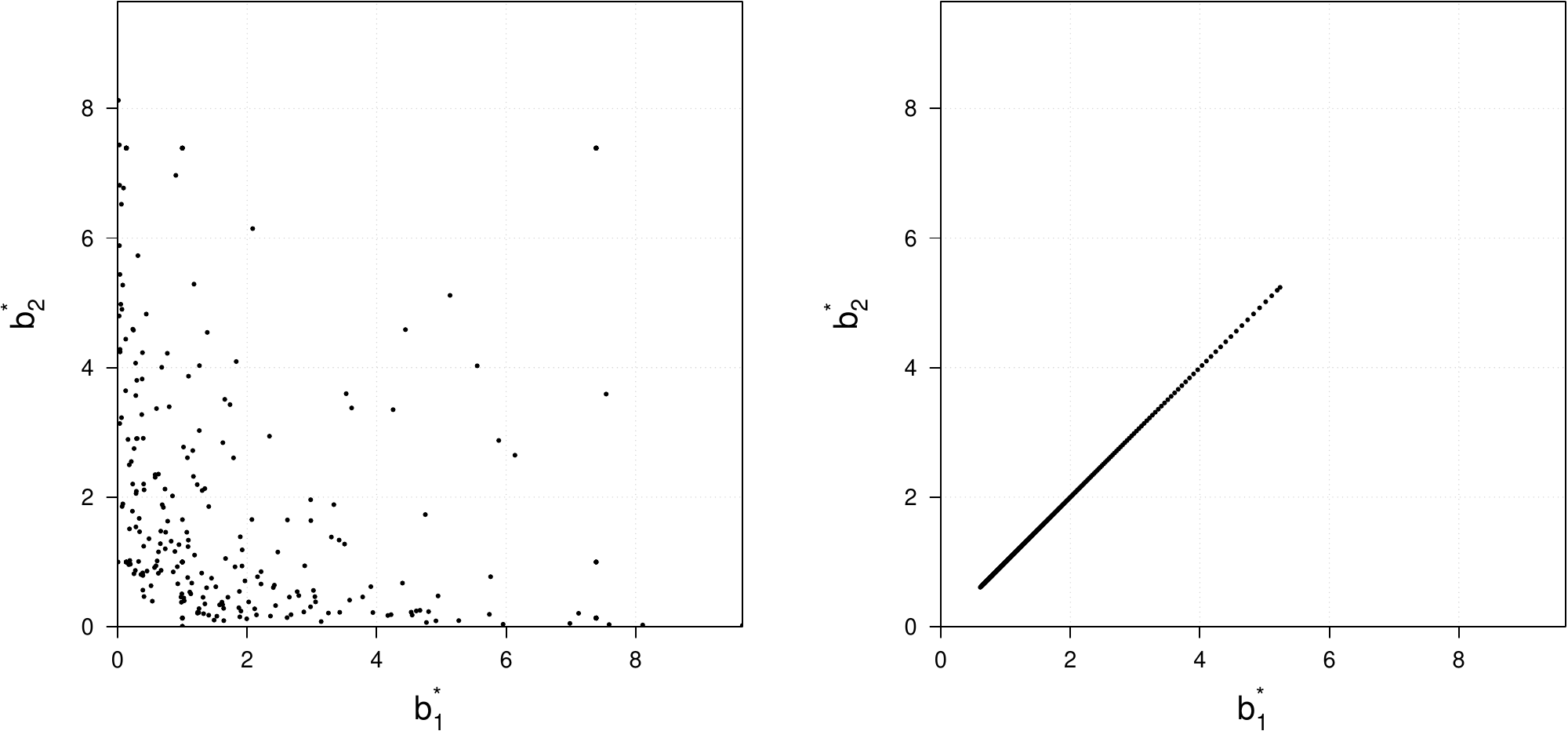}

    \caption{Numerically computed vector
  $(b_1^\ast,b_2^\ast)$, with $b_3^\ast=1/(b_1^\ast b_2^\ast)$,
  for three-dimensional $t$ copulas with $\nu=5$ degrees of freedom.
  The left panel uses $300$ randomly generated correlation matrices.
  The right panel fixes $r_{1,3}=r_{2,3}=0.4$ and varies $r_{1,2}$
over $300$ equally spaced values within the interval for which
the resulting correlation matrix is positive definite.
  }
    \label{fig:bast:t}
  \end{figure}
\end{Example}

\subsection{Survival Marshall--Olkin copulas}
We consider a parametric subclass of EVCs with stable tail dependence functions
\begin{align*}
  \ell_{\bm\alpha}(\bx)
  =
  \sum_{j=1}^d (1-\alpha_j)x_j+\max_{1\le j\le d}\{\alpha_j x_j\},
  \quad \bx\in[0,\infty)^d,
\end{align*}
where $\alpha_1,\dots,\alpha_d\in(0,1]$.
The corresponding EVC is the MO copula given by
\begin{align}
  C^{\mathrm{MO}}_{\bm\alpha}(\bm u)
  =
  \biggl(\,\prod_{j=1}^d u_j^{1-\alpha_j}\biggr)\min_{1\le j\le d}\{u_j^{\alpha_j}\},
  \quad \bm u\in[0,1]^d.
  \label{eq:mo:copula:explicit}
\end{align}
Note that this is a special case of the construction of \cite{liebscher2008} based on the independence and comonotone copulas; see
also~\citet[][Equation~(2)]{mai2010pickands} for the case of a single common
shock affecting all components.

The next proposition shows that, for this family, both the MTCM and the
attaining $\bm b^\ast$ are available in closed form.
\begin{Proposition}[MTCM for survival Marshall--Olkin copulas]\label{prop:mo:mtcm}
  Let $\hat C^{\mathrm{MO}}_{\bm\alpha}$ be the survival copula of~\eqref{eq:mo:copula:explicit}
  for $\bm\alpha=(\alpha_1,\dots,\alpha_d)\in(0,1]^d$. Then the following statements hold:
  \begin{enumerate}[label=(\roman*), labelwidth=\widthof{(iii)}]
  \item
    $\hat C^{\mathrm{MO}}_{\bm\alpha}$ admits the tail copula $\Lambda(\bx;\hat C^{\mathrm{MO}}_{\bm\alpha})
      =
      \min_{1\le j\le d}\{\alpha_j x_j\}$, $\bx\in(0,\infty)^d$.
  \item
    Its MTCM is given by $\lambda^\ast(\hat C^{\mathrm{MO}}_{\bm\alpha})
      =
      \prod_{j=1}^d \alpha_j^{1/d}$.
  \item
    The maximizer in~\eqref{eq:def:mtcm:multivariate} is unique and equals
    \begin{align*}
      \bm b^\ast(\hat C^{\mathrm{MO}}_{\bm\alpha})
      =
      \Biggl(\frac{\prod_{j=1}^d \alpha_j^{1/d}}{\alpha_1},
      \dots,
      \frac{\prod_{j=1}^d \alpha_j^{1/d}}{\alpha_d}
      \Biggr)\in\mathcal B.
    \end{align*}
  \end{enumerate}
\end{Proposition}

\section{Archimax and nested Archimedean copulas}\label{sec:archimax:hac}
In this section, we derive the MTCM and its maximizer for Archimax and
nested Archimedean copulas.

\subsection{MTCM for Archimax copulas}
A function $\psi:[0,\infty)\to [0,1]$ is an \emph{Archimedean generator}
if it is continuous, decreasing with $\psi(0)=1$ and $\psi(\infty)\coloneq\lim_{t\to\infty}\psi(t)=0$,
and strictly decreasing on $[0,\inf\{t:\psi(t)=0\})$. The set of all such
$\psi$ is denoted by $\Psi$.
A $d$-dimensional copula $C$ is an \emph{Archimedean copula} if it is of the form
\begin{align*}
  C_\psi(\bm{u})=\psi(\psii(u_1)+\dots+\psii(u_d)),\quad\bm{u}\in[0,1]^d,
\end{align*}
for a generator $\psi\in\Psi$ with inverse $\psii:[0,1]\to [0,\infty]$ where
$\psii(0)=\inf\{t:\psi(t)=0\}$. According to \cite{malov2001} and
\cite{mcneilneslehova2009}, the function $C_\psi$ is a $d$-dimensional copula if and only if
$\psi$ is \emph{$d$-monotone}, i.e.\ $\psi$ is continuous on $[0,\infty)$,
admits derivatives up to the order $d-2$ satisfying $(-1)^k\psi^{(k)}(t)\ge 0$
for all $k\in\{0,\dots,d-2\}$, $t\in(0,\infty)$, and $(-1)^{d-2}\psi^{(d-2)}(t)$
is decreasing and convex on $(0,\infty)$. The set of all $d$-monotone
Archimedean generators is denoted by $\Psi_d$. An Archimedean generator $\psi$
is \emph{completely monotone} if it is $d$-monotone for every integer $d\ge 2$, i.e.\ $\psi$ is infinitely differentiable on $(0,\infty)$ and
$(-1)^k\psi^{(k)}(t)\ge0$ for every $k=0,1,\dots$ and $t>0$.
The set of all completely monotone Archimedean generators is denoted
by $\Psi_\infty$.

For an Archimedean generator $\psi\in\Psi_d$ and a stable tail dependence function
$\ell$, the \emph{Archimax copula} generated by $(\psi,\ell)$ is
\begin{align}
  C_{\psi,\ell}(\bm u)=\psi\bigl(\ell(\psi^{-1}(u_1),\dots,\psi^{-1}(u_d))\bigr),\quad \bm u\in[0,1]^d;\label{eq:def:archimax}
\end{align}
see \citet{charpentierfougeresgenestneslehova2014}. As a special case, if
$\ell(\bx)=\ell_\Pi(\bx):=\sum_{j=1}^d x_j$, then \eqref{eq:def:archimax} reduces to the
Archimedean copula with generator $\psi\in\Psi_d$.
As another special case, if $\psi(t)=e^{-t}$,
then~\eqref{eq:def:archimax} reduces to the standard representation~\eqref{eq:EVC:form}
of an EVC.

We can now provide the tail copula of Archimax copulas with regularly varying
generator $\psi\in\Psi_{\infty}$. To this end, a measurable function
$f:(0,\infty)\to (0,\infty)$ is called \emph{regularly varying} (at $\infty$)
with index $\rho\in\R$ if $\lim_{x\to\infty} f(tx)/f(x)= t^\rho$
for all $t>0$.  Denote by $\mathrm{RV}_{\rho}$ the class of all regularly
varying functions with index $\rho\in\R$.

\begin{Proposition}[Tail copula of Archimax copulas]\label{prop:tail:cop:AC:reg:var:gen}
  Let $\psi\in\Psi_{\infty}\cap\mathrm{RV}_{-\alpha}$ for some $\alpha>0$, and
  let $\ell$ be a stable tail dependence function. Then the tail copula of the
  Archimax copula $C_{\psi,\ell}$ generated by $(\psi, \ell)$ is
  \begin{align*}
    \Lambda(\bx;C_{\psi,\ell})=\bigl\{\ell\bigl(x_1^{-1/\alpha},\dots,x_d^{-1/\alpha}\bigr)\bigr\}^{-\alpha},\quad \bx\in(0,\infty)^d.
  \end{align*}
\end{Proposition}
For Archimedean copulas (the special case of $\ell=\ell_\Pi$), we have
\begin{align}
  \Lambda(\bx; C_{\psi,\ell_\Pi})=\biggl(\,\sum_{j=1}^d x_j^{-1/\alpha}\biggr)^{-\alpha},\quad \bx\in (0,\infty)^d.\label{eq:S:alpha}
\end{align}

The following example shows that various Archimedean generators are regularly
varying and thus the corresponding Archimedean copulas admit tail copulas of
Form~\eqref{eq:S:alpha}.

\rev{
\begin{Example}[Regularly varying Archimedean generators]
  \leavevmode
  Let $\psi\in\Psi_\infty\cap\mathrm{RV}_{-\alpha}$
for some $\alpha>0$.
The following transformations preserve complete monotonicity
and regular variation under the conditions stated below.
  \begin{enumerate}[label=\arabic*), labelwidth=\widthof{1)}]
  \item
Consider the \emph{inner power generator}
\citep{nelsen2006introduction,hofert2010c}
$\tilde\psi(t)=\{\psi(t)\}^{1/\gamma}$,
$\gamma\in(0,1]$.
Assume that $\tilde\psi$ is completely monotone,
as is automatically the case when $1/\gamma$ is a positive integer.
Then it is straightforward to check that $\tilde\psi\in\mathrm{RV}_{-\alpha/\gamma}$.
  \item The \emph{outer power generator}
    \citep{nelsen2006introduction, hofert2010c,hofert2011a} 
    $\tilde{\psi}(t)=\psi(t^{1/\beta})$, $t\ge 0$, $\beta> 1$, satisfies $\tilde{\psi}\in\mathrm{RV}_{-\alpha/\beta}$.
See~\citet{hofertscherer2011} for an application of outer power
    Clayton copulas.
    \item
  The \emph{tilted outer power generator}
  \citep{hofert2010c,hofert2011a} is
$\tilde\psi(t)
    =
    \psi\bigl((c^\beta+t)^{1/\beta}-c\bigr)$, $t\geq0$, $\beta\geq1$, $c\geq0$.
  Put $g(x)=(c^\beta+x)^{1/\beta}-c$.
  For every fixed $t>0$, both $g(x)$ and $g(tx)$ tend to infinity,
  and $g(tx)\simeq g(x)t^{1/\beta}$ as $x\to \infty$.
  Thus, applying~\citet[Proposition~2.6~(iii)]{resnick2007heavy} yields
  \begin{align*}
    \frac{\tilde\psi(tx)}{\tilde\psi(x)}
    =
    \frac{\psi(g(tx))}{\psi(g(x))}
    \rightarrow
    \bigl(t^{1/\beta}\bigr)^{-\alpha}
    =
    t^{-\alpha/\beta}.
  \end{align*}
  Hence $\tilde\psi\in\mathrm{RV}_{-\alpha/\beta}$.

\item
  The \emph{shifted generator}
  \citep{hofert2010c,hofert2011a} is
  $\tilde\psi(t)=\psi(t+h)/\psi(h)$, $t\geq0$, $h\geq0$.
  For every fixed $t>0$, applying~\citet[Proposition~2.6~(iii)]{resnick2007heavy} again leads to  $\tilde\psi(tx)/\tilde\psi(x)
    =
    \psi(tx+h)/\psi(x+h)
    \rightarrow t^{-\alpha}$.
  Hence $\tilde\psi\in\mathrm{RV}_{-\alpha}$ for every $h\geq0$.
  \end{enumerate}
\end{Example}
}

Based on Proposition~\ref{prop:tail:cop:AC:reg:var:gen}, we can now provide the MTCM of Archimax copulas.

\begin{Theorem}[MTCM for Archimax copulas]\label{thm:mtcm:archimax:minimization}
  Let $C_{\psi,\ell}$ be the Archimax copula generated by $\psi\in\Psi_\infty\cap\mathrm{RV}_{-\alpha}$, $\alpha>0$,
  and the stable tail dependence function $\ell$. Then
  \begin{align*}
    \lambda^\ast(C_{\psi,\ell})
    = \max_{\bm b\in\mathcal B}
    \ell(b_1^{-1/\alpha},\dots,b_d^{-1/\alpha})^{-\alpha}
    = \bigl\{\min_{\bm z\in\mathcal B}\ell(\bm z)\bigr\}^{-\alpha}.
  \end{align*}
  Moreover, the following statements are equivalent:
  \begin{enumerate}[label=(\roman*), labelwidth=\widthof{(ii)}]
  \item $\Lambda(\cdot;C_{\psi,\ell})$ is uniquely maximized at  $\bm b^\ast\in\mathcal B$.
  \item The function $\ell$ on $\mathcal B$ is uniquely minimized at $\bm z^\ast=((b_1^{\ast})^{-1/\alpha},\dots,(b_d^{\ast})^{-1/\alpha})\in\mathcal B$.
  \end{enumerate}
\end{Theorem}

For Archimax copulas, exchangeability of $\ell$ already implies that the
maximizer is unique and satisfies $\bm b^\ast=\bone_d=(1,\dots,1)$.
\begin{Corollary}[Archimax copulas under exchangeability]\label{cor:mtcm:archimax:symmetric}
  Under the assumptions of Theorem~\ref{thm:mtcm:archimax:minimization},
  assume in addition that $\ell$ is exchangeable. Then $\lambda^\ast(C_{\psi,\ell})=\ell(\bone_d)^{-\alpha}$
  and the maximizer is uniquely given by $\bm b^\ast=\bone_d$.
\end{Corollary}

\subsection{Nested Archimedean copulas}
As a generalization of Archimedean copulas, we now consider \emph{nested Archimedean copulas}, which are
Archimedean copulas with arguments possibly replaced by (nested) Archimedean
copulas. For a description of specific members of this class of copulas, we
need some notions of the realm of graph theory.

Let $\mathcal T=(\mathcal V,\mathcal E)$ be an undirected, rooted tree with a
set of vertices $\mathcal V$ and a set of edges
$\mathcal E\subseteq\mathcal V\times \mathcal V$.  Denote by $r$ the root, by
$\mathcal L$ the set of leaves and by
$\mathcal I:=\mathcal V\setminus\mathcal L$ the set of internal vertices of
$\mathcal T$. 
\rev{We identify the leaf set $\mathcal L$ with $\{1,\ldots,d\}$ by an arbitrary fixed labeling, where $d=|\mathcal L|$.}
For $v\in \mathcal V$, let $\mathrm{pa}(v)$ denote the
direct 
parent vertex of $v$ and $\mathrm{ch}(v)$ the set of direct children of $v$,
with $\mathrm{pa}(r)=\emptyset$ and $\mathrm{ch}(v)=\emptyset$,
$v \in \mathcal L$. Let $\mathrm{le}(v)\subseteq\mathcal L$ denote the set of
all leaves of $v$ and $d(v):=|\mathrm{le}(v)|$ its cardinality, with
$\mathrm{le}(v)=\{v\}$ and thus $d(v)=1$ for all $v\in\mathcal
L$. 
\rev{
For $v\in\mathcal V$, let
$\mathcal T_v=(\mathcal V_v,\mathcal E_v)$ be the subtree induced
by $v$ and all its descendants, and let $\mathcal I_v$ denote
its set of internal vertices.
For $w\in\mathcal V_v$, let $\mathrm{an}_v(w)$ denote the set
of ancestors of $w$, which are vertices strictly between $v$ and $w$ on the unique path
from $v$ to $w$.
Thus both endpoints are excluded, and
$\mathrm{an}_v(w)=\emptyset$ if $w=v$ or
$w\in\mathrm{ch}(v)$.
}

Based on the set $\Psi_{\mathcal T}:=\{\psi_v\}_{v\in\mathcal I}\subset \Psi_\infty$ of Archimedean
generators, a nested Archimedean copula can now be fully
described by $(\mathcal T,\Psi_{\mathcal T})$.
We associate each vertex $v\in \mathcal V$ with a $[0,1]$-valued function:
\begin{align*}
  C_v(\bm u_{\mathrm{le}(v)})=\begin{cases}
    u_v,& v\in\mathcal L,\\
    \psi_v\bigl(\sum_{w\in\mathrm{ch}(v)}\psi_v^{-1}(C_w(\bm u_{\mathrm{le}(w)}))\bigr),& v\in\mathcal I, 
  \end{cases}
  \qquad \bu_{\mathrm{le}(v)}\in[0,1]^{d(v)}.
\end{align*}
which is defined recursively.
We call $C_v$ the marginal nested Archimedean copula associated with $v$ and $C_{(\mathcal T,\Psi_{\mathcal T})}:=C_r$ the \emph{nested Archimedean copula}
associated with $(\mathcal T,\Psi_{\mathcal T})$.

To guarantee that the resulting nested construction is indeed a proper copula, we
assume the following \emph{sufficient nesting condition} of \cite{mcneil2008}:
\begin{align*}
  (\psi_{\mathrm{pa}(v)}^{-1}\circ\psi_v)'\ \text{is completely monotone for every}\ v\in\mathcal I\setminus\{r\}.
\end{align*}
In addition, we say that a nested Archimedean copula
$C_{(\mathcal T,\Psi_{\mathcal T})}$ has \emph{regularly varying generators} if
for every $v \in \mathcal I$ there exists $\alpha_v>0$ such that
$\psi_v\in\Psi_\infty\cap\mathrm{RV}_{-\alpha_v}$.

We can now describe the tail copula of a nested Archimedean copula $C_{(\mathcal T,\Psi_{\mathcal T})}$
with regularly varying generators.

\begin{Proposition}[Tail copula of a nested Archimedean copula]\label{prop:tail:cop:tree:nested:archimedean}
  \rev{
  Let $C_{(\mathcal T,\Psi_{\mathcal T})}$ be a nested Archimedean copula with $\psi_v \in \Psi_\infty \cap \mathrm{RV}_{-\alpha_v}$ for some $(\alpha_v)_{v \in \mathcal I} \in (0,\infty)^{\mathcal I}$.
  }
  Recursively define the functions by
  \begin{align}\label{eq:recursion:tail:copula:tree}
    \Lambda_v(\bx_{\mathrm{le}(v)})=\begin{cases}
      x_v,& v \in \mathcal L,\\
      \bigl\{\sum_{w\in\mathrm{ch}(v)} \Lambda_w(\bx_{\mathrm{le}(w)})^{-1/\alpha_v}\bigr\}^{-\alpha_v},& v\in\mathcal I,
    \end{cases}
    \qquad \bx_{\mathrm{le}(v)}\in(0,\infty)^{d(v)}.
  \end{align}
  Then $\Lambda(\cdot;C_v)=\Lambda_v(\cdot)$ for all $v \in \mathcal V$.
\end{Proposition}

Under the same setup, the MTCM of a nested Archimedean copula
$C_{(\mathcal T,\Psi_{\mathcal T})}$ is given as follows.

\begin{Theorem}[MTCM for nested Archimedean copulas]\label{thm:mtcm:nested:archimedean:recursive}
  \rev{
  Let $C_{(\mathcal T,\Psi_{\mathcal T})}$ be a nested Archimedean copula with $\psi_v \in \Psi_\infty \cap \mathrm{RV}_{-\alpha_v}$ for some $(\alpha_v)_{v \in \mathcal I} \in (0,\infty)^{\mathcal I}$.
  }
  Then the following statements hold.
  Below, an empty product is understood as $1$.
  \begin{enumerate}[label=(\roman*)]
  \item\label{prop:mtcm:nested:archimedean:recursive:i} Recursively define
    \begin{align}
      \lambda_v^\ast
      =
      \begin{cases}
        1,
        & v\in\mathcal L,\\
        d(v)^{-\alpha_v}
        \displaystyle
        \prod_{w\in\mathrm{ch}(v)}
        \{d(w)^{\alpha_v}\lambda_w^\ast\}^{d(w)/d(v)},
        & v\in\mathcal I.
      \end{cases}
      \label{eq:recursion:lambda:tree}
    \end{align}
    Then $\lambda^\ast(C_v)=\lambda_v^\ast$ for all $v\in\mathcal I$.
  \item For all $v \in \mathcal I$, the MTCM of $C_v$ admits the closed form
    \begin{align}
      \lambda_v^\ast
      =
      d(v)^{-\alpha_v}
      \prod_{w\in\mathcal I_v\setminus\{v\}}
      d(w)^{
      (\alpha_{\mathrm{pa}(w)}-\alpha_w)d(w)/d(v)
      }.
      \label{eq:closed:form:lambda:subtree}
    \end{align}

  \item For all $v \in \mathcal I$, the maximizer
    $
    \bm b_v^\ast
    $
    of the MTCM of $C_v$ is unique and given by
    \begin{align}
     (\bm b_v^\ast)_j
      =
      \lambda_v^\ast\ d(v)^{\alpha_v}\!\!\!\prod_{w\in\mathrm{an}_v(j)}
      d(w)^{\alpha_w-\alpha_{\mathrm{pa}(w)}},\qquad j\in\mathrm{le}(v).
    \label{eq:closed:form:b:tree}
    \end{align}
  \end{enumerate}
\end{Theorem}

\rev{
\begin{Remark}[Interpretation of Theorem~\ref{thm:mtcm:nested:archimedean:recursive}]
Fix $v\in\mathcal I$ and define 
\begin{align*}
    z_j(v) &:= \sum_{w \in \operatorname{an}_v(j)}\bigl(\alpha_w - \alpha_{\operatorname{pa}(w)}\bigr)\log(d(w)),\\
    \overline{z}(v) &:= \frac{1}{d(v)}\sum_{j \in \operatorname{le}(v)}z_j(v)=
\sum_{w\in\mathcal I_v\setminus\{v\}}
\frac{d(w)\log (d(w))}{d(v)}
\bigl(\alpha_w-\alpha_{\operatorname{pa}(w)}\bigr).
\end{align*}
Taking logarithms in~\eqref{eq:closed:form:b:tree} yields $\log (\bm b_v^\ast)_j = \log \big(\lambda_v^\ast d(v)^{\alpha_v}\big) + z_j(v)$.
 Since $\sum_{j \in \operatorname{le}(v)} \log(\bm b_v^\ast)_j = 0$, we obtain 
 \begin{align*}
 \log(\bm b_v^\ast)_j = z_j(v)-\overline{z}(v)\qquad\text{and}\qquad
\log \lambda_v^\ast
=-\alpha_v \log d(v)-\overline{z}(v).
 \end{align*}
 Thus, $z_j(v)-\overline{z}(v)$ is the signed logarithmic deviation
of the $j$th coordinate of $\bm b_v^\ast$ from the
diagonal benchmark $\bm 1_{d(v)}$.
For $\log \lambda_v^\ast$, the first term $-\alpha_v \log d(v)$ is the logarithm of the MTCM of the unnested $d(v)$-dimensional Archimedean copula generated by $\psi_v$, whereas the correction term $-\overline{z}(v)$ is a weighted sum
of the parent--child parameter differences
$\alpha_{\operatorname{pa}(w)}-\alpha_w$ with weights $d(w)\log (d(w))/d(v)$.
\end{Remark}
}

\begin{Example}[Two-level nested Archimedean copulas]\label{ex:two:level}
To construct a $d$-dimensional nested Archimedean copula, write $d=\sum_{s=1}^S d_s$, where $d_s\ge 1$.
For $s\in\{1,\dots,S\}$ with $d_s\ge 2$, let $C_s$ be the $d_s$-dimensional Archimedean copula with generator $\psi_s\in\Psi_\infty\cap\mathrm{RV}_{-\alpha_s}$, $\alpha_s>0$.
For $s\in \{1,\dots,S\}$ with $d_s=1$, write $C_s(u_s)=u_s$.
Denote by $C_0$ the $S$-dimensional Archimedean copula with generator $\psi_0\in\Psi_\infty\cap\mathrm{RV}_{-\alpha_0}$, $\alpha_0>0$.
Then the nested Archimedean copula corresponding to this tree structure is
\begin{align*}
    C(\bm u)=C_0(C_1(\bm u_1),\dots,C_S(\bm u_S)),
    \quad \bm u=(\bm u_1,\dots,\bm u_S)\in[0,1]^d,
  \end{align*}
  where $\bm u_s\in[0,1]^{d_s}$, $s=1,\dots,S$.
  By Theorem~\ref{thm:mtcm:nested:archimedean:recursive}, we can directly calculate the MTCM $\lambda^\ast(C)$ and its unique maximizer $\bm{b}^\ast=(b_{s,j}^\ast)_{s=1,\dots,S,\,j=1,\dots,d_s}$ as follows:
  \begin{align*}
       \lambda^\ast(C)
    &= d^{-\alpha_0}
    \prod_{s=1}^S d_s^{d_s(\alpha_0-\tilde\alpha_s)/d},\\
    b_{s,j}^\ast
    &=
    d_s^{\tilde\alpha_s-\alpha_0}
    \,\prod_{t=1}^S d_t^{d_t(\alpha_0-\tilde\alpha_t)/d}, \quad s=1,\dots,S,\ \ j=1,\dots,d_s,
  \end{align*}
  where $\tilde \alpha_s=\alpha_0$ if $d_s=1$ and $\tilde \alpha_s=\alpha_s$ if $d_s\ge 2$.
\end{Example}

\begin{Example}[Illustrative examples]\label{ex:illustrative}
  Consider two-level nested Archimedean copulas as in the setup of Example~\ref{ex:two:level}.
 We take Clayton generators, which are regularly varying with index $-\alpha_s=-1/\theta_s$, $s\in\{0,\dots,S\}$.
  For these generators, the sufficient nesting conditions are $\theta_{0} \le \theta_s$, $s=1,\dots,S$, or equivalently, $\alpha_0\ge \max\{\alpha_1,\dots,\alpha_S\}$~\citep{joe1997multivariate,mcneil2008}.
  \rev{Figure~\ref{fig:diagrams} shows the tree diagrams of the following two models.}
  \begin{figure}[tbp]
    \centering
    \begin{tikzpicture}[
      grow=south,
      level 1/.style={sibling distance=36mm,level distance=11mm},
      level 2/.style={sibling distance=13mm,level distance=11mm},
      edge from parent/.style={thick,draw=black!40,shorten >=1mm,shorten <=1mm},
      every node/.style={rectangle,rounded corners,shade,top color=white,bottom color=black!20,draw=black!40,thick,inner sep=1.5mm},
      mynodestyle/.style={minimum width=12mm,minimum height=6mm},
      myleafstyle/.style={minimum width=7mm,minimum height=5mm}
      ]
      \node[mynodestyle]{$C_0$}
      child[edge from parent/.style={thick,draw=black,shorten >=1mm,shorten <=1mm}]{
        node[mynodestyle]{$C_1$}{
          child[edge from parent/.style={thick,draw=black,shorten >=1mm,shorten <=1mm}]{
            node[myleafstyle]{$u_{1,1}$}
          }
          child[edge from parent/.style={thick,draw=black,shorten >=1mm,shorten <=1mm}]{
            node[myleafstyle]{$u_{1,2}$}
          }
        }
      }
      child[edge from parent/.style={thick,draw=black,shorten >=1mm,shorten <=1mm}]{
        node[mynodestyle]{$C_2$}{
          child[edge from parent/.style={thick,draw=black,shorten >=1mm,shorten <=1mm}]{
            node[myleafstyle]{$u_{2,1}$}
          }
          child[edge from parent/.style={thick,draw=black,shorten >=1mm,shorten <=1mm}]{
            node[myleafstyle]{$u_{2,2}$}
          }
          child[edge from parent/.style={thick,draw=black,shorten >=1mm,shorten <=1mm}]{
            node[myleafstyle]{$u_{2,3}$}
          }
        }
      };
    \end{tikzpicture}
    \qquad
    \begin{tikzpicture}[
      grow=south,
      level 1/.style={sibling distance=26mm,level distance=11mm},
      level 2/.style={sibling distance=16mm,level distance=11mm},
      edge from parent/.style={thick,draw=black!40,shorten >=1mm,shorten <=1mm},
      every node/.style={rectangle,rounded corners,shade,top color=white,bottom color=black!20,draw=black!40,thick,inner sep=1.5mm},
      mynodestyle/.style={minimum width=12mm,minimum height=6mm},
      myleafstyle/.style={minimum width=7mm,minimum height=5mm}
      ]
      \node[mynodestyle]{$C_0$}
      child[edge from parent/.style={thick,draw=black,shorten >=1mm,shorten <=1mm}]{
        node[myleafstyle]{$u_{0,1}$}
      }
      child[edge from parent/.style={thick,draw=black,shorten >=1mm,shorten <=1mm}]{
        node[mynodestyle]{$C_1$}{
          child[edge from parent/.style={thick,draw=black,shorten >=1mm,shorten <=1mm}]{
            node[myleafstyle]{$u_{1,1}$}
          }
          child[edge from parent/.style={thick,draw=black,shorten >=1mm,shorten <=1mm}]{
            node[myleafstyle]{$u_{1,2}$}
          }
        }
      };
    \end{tikzpicture}
    \caption{\rev{Tree diagrams of the models in Example~\ref{ex:illustrative}}.}
    \label{fig:diagrams}
  \end{figure}
  \begin{enumerate}[label=(\roman*)]
  \item If $C(\bm{u})=C_0(C_1(u_{1,1}, u_{1,2}),C_2(u_{2,1},u_{2,2},u_{2,3}))$, we obtain
    \begin{align*}
      \lambda^\ast(C)=\frac{5^{-\alpha_0}}{2^{2(\alpha_1-\alpha_0)/5}\cdot 3^{3(\alpha_2-\alpha_0)/5}},
    \end{align*}
    with maximizing $\bm{b}^{\ast}=(b_{1,1}^{\ast},b_{1,2}^{\ast},b_{2,1}^{\ast},b_{2,2}^{\ast},b_{2,3}^{\ast})$ given by
    \begin{align*}
      b_{1,1}^{\ast}=b_{1,2}^{\ast}=\frac{2^{\alpha_1-\alpha_0}}{2^{2(\alpha_1-\alpha_0)/5}3^{3(\alpha_2-\alpha_0)/5}},\quad b_{2,1}^{\ast}=b_{2,2}^{\ast}=b_{2,3}^{\ast}=\frac{3^{\alpha_2-\alpha_0}}{2^{2(\alpha_1-\alpha_0)/5}3^{3(\alpha_2-\alpha_0)/5}}.
    \end{align*}
  \item If $C(\bm{u})=C_0(u_{0,1}, C_1(u_{1,1}, u_{1,2}))$, then $\lambda^\ast(C)=3^{-\alpha_0}/2^{2(\alpha_1-\alpha_0)/3}$
    with maximizing $\bm{b}^{\ast}=(b_{0,1}^{\ast},b_{1,1}^{\ast},b_{1,2}^{\ast})$ given by $b_{0,1}^{\ast}=2^{2(\alpha_0-\alpha_1)/3}$ and $b_{1,1}^{\ast}=b_{1,2}^{\ast}=2^{(\alpha_1-\alpha_0)/3}$.
    Figure~\ref{fig:MTCM:3d:fully:nested:reg:var:gen} shows the MTCM
    $\lambda^\ast(C)$ (left), maximizing $b_{0,1}^{\ast}$ (center) and $b_{1,1}^{\ast}=b_{1,2}^{\ast}$ (right)
    as functions of $\theta_s=1/\alpha_s$, $s=0,1$.
    \begin{figure}[tbp]
      \centering
      \includegraphics[width=0.32\textwidth]{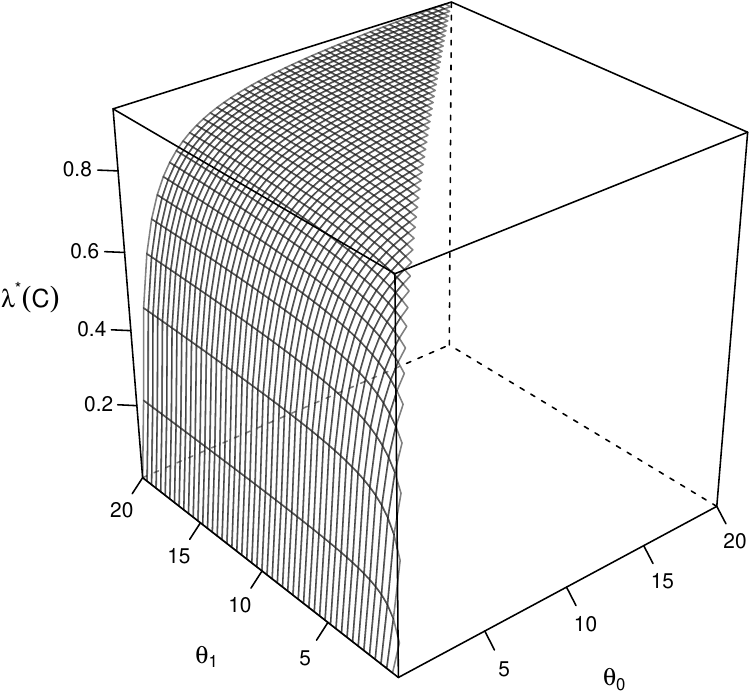}\hfill
      \includegraphics[width=0.32\textwidth]{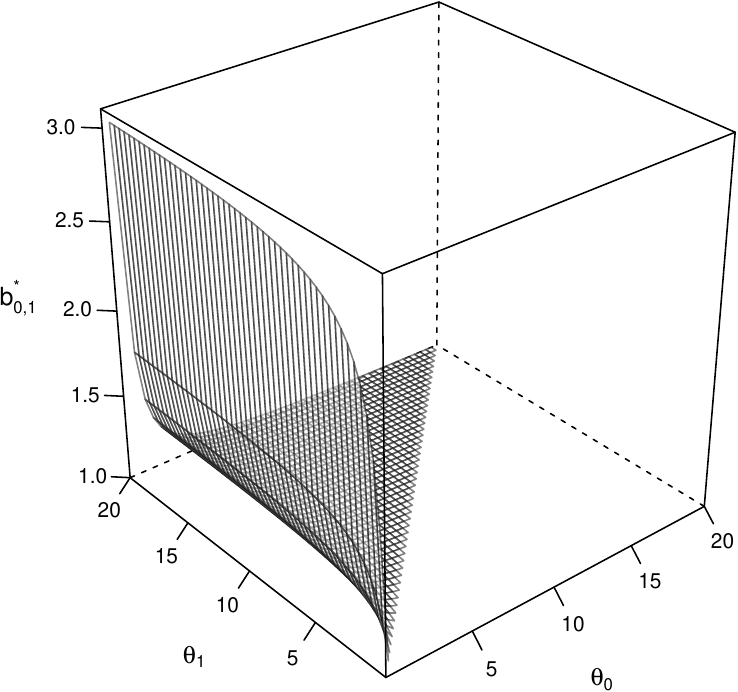}\hfill
      \includegraphics[width=0.32\textwidth]{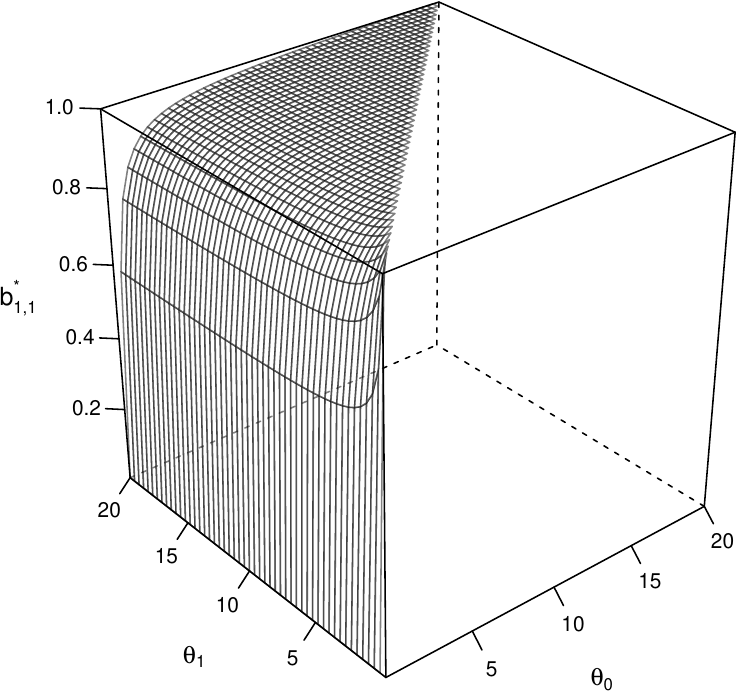}
      \caption{MTCM $\lambda^\ast(C)$ (left) and attaining $b_{0,1}^{\ast}$
        (center) and $b_{1,1}^{\ast}$ (right) as functions of Clayton parameters
        $\theta_0$ and $\theta_1$ (with $0<\theta_0\le\theta_1$), where
        $\theta_0=1/\alpha_0$ and $\theta_1=1/\alpha_1$.  Note that a larger
        $\theta$ indicates larger concordance.}
\label{fig:MTCM:3d:fully:nested:reg:var:gen}
    \end{figure}
  \end{enumerate}
\end{Example}

\section{Application to sea-level models and flood risk}\label{sec:numerical:real:data}

\subsection{Models and tail-dependence comparison}\label{sec:model:comparison}

\citet{smith1990statistics} and~\citet{tawn1990modelling} analyze 40 years of
trivariate annual sea-level maxima at the three sites Southend (dimension
$j=1$), Sheerness (dimension $j=2$) and Kings Lynn (dimension $j=3$) on the
south-east coast of England. Trivariate extreme value distributions are fitted
to the data using the maximum likelihood method.
As an
application of MTCMs, this section revisits the models fitted in~\citet{tawn1990modelling}
to assess the risk of simultaneous high sea levels.

Let $\bm{Y}=(Y_1,Y_2,Y_3)\sim H$ be the random vector of annual maxima, whose
marginal distributions $H_1,H_2,H_3$ are generalized extreme value distributions.
Let
$X_j=-\log H_j(Y_j)$, $j=1,2,3$, so that
$\bX=(X_1,X_2,X_3)\sim G$ has standard exponential margins. Note that the upper joint tail of
$\bm{Y}$, modeling jointly large sea levels and thus corresponding to an
increased risk of a simultaneous flooding at the three sites, is transformed to
the lower tail of $\bX$, which is suitable for our analysis as in Section~\ref{sec:copulas:doa}.
The model considered in~\citet{tawn1990modelling} is based on the survival function
$\bar G$ of $G$ given by
\begin{align*}
  \bar G(\bx)=\exp\biggl(-(x_1+x_2+x_3) B\biggl(\frac{x_1}{x_1+x_2+x_3},\frac{x_2}{x_1+x_2+x_3}\biggr)\biggr),\quad \bx=(x_1,x_2,x_3)\in(0,\infty)^3,
\end{align*}
where $B$ is a function of a specific form on the simplex
$\Delta_2=\{(w_1,w_2)\in[0,\infty)^2:w_1+w_2\le 1\}$;
see~\citet[Equation~(2$\cdot$4)]{tawn1990modelling}.
This corresponds to assuming an EVC on $\bm{Y}$ with the stable tail dependence function
\begin{align*}
  \ell(x_1,x_2,x_3)=(x_1+x_2+x_3) B\biggl(\frac{x_1}{x_1+x_2+x_3},\frac{x_2}{x_1+x_2+x_3}\biggr),
\end{align*}
which can be found by comparing $\bar G$
with~\eqref{eq:EVC:form}.

With $w_3=1-w_1-w_2$, \citet{tawn1990modelling} considered the two models
\begin{align*}
  B_{\one}(w_1,w_2)&=(1-\theta_3) w_3 + \biggl[\,\sum_{j=1}^2 \{
                     (1-\theta_j)w_j
                     \}^r
                     \biggr]^{1/r}
                     +
                     \biggl\{\,
                     \sum_{j=1}^3(
                     \theta_j w_j
                     )^s
                     \biggr\}^{1/s},\\
  B_{\two}(w_1,w_2)&=\phi (
                     (w_1^{rs}+w_2^{rs})^{1/r}
                     +w_3^s
                     )^{1/s}
                     +(1-\phi)\{
                     (w_1^t+w_2^t)^{1/t}+
                     w_3\},
\end{align*}
for different choices of $r,s,t\ge 1$ and $\theta_1,\theta_2,\theta_3,\phi\in[0,1]$. 
\citet{tawn1990modelling} considered three models of Type~$\one$ (labeled
$\one$-1, $\one$-2, $\one$-3), two models of Type~$\two$ (labeled $\two$-1, $\two$-2)
and the independence model.
Note that Type $\one$-1 corresponds to the
symmetric Gumbel copula.
Table~\ref{tab:1:tawn} lists the fitting results for
the five parametric models among these six models.
\begin{table}[ht]
  \centering
  \begin{tabular}{lccc}
    \toprule
    Model & Constraints & Log-likelihood & Parameter estimates \\
    \midrule
    Type $\one$-1  & $\theta_1=\theta_2=\theta_3=1$ & $-95.93$ & $s=1.59$ \\
    Type $\one$-2 & $\theta_1=\theta_2=1$ & $-88.85$ & $(s,\theta_3)=(2.48,0.25)$ \\
    Type $\one$-3 & $\theta_1=\theta_2=\theta$  & $-86.15$ & $(s,r,\theta,\theta_3)=(7.44,2.21,0.23,0.55)$\\[2mm]
    Type $\two$-1 & $\phi=1$ & $-89.26$ & $(s,r)=(1.59,1.27)$ \\
    Type $\two$-2 & -- & $-86.09$ & $(s,r,t,\phi)=(1.69,1.25,7.44,0.74)$ \\
    \bottomrule
  \end{tabular}
  \caption{Models, constraints and parameter estimates
    of~\citet[Section~5]{tawn1990modelling}.}\label{tab:1:tawn}
\end{table}

We now compare the five models listed in Table~\ref{tab:1:tawn} in terms of
their implied MTCM.  For each model, we compute the TDC
$\lambda(\hat{C}_\ell)=\Lambda(1,1,1)$, the MTCM $\lambda^\ast(\hat{C}_\ell)$ and its maximizer $\bm b^\ast\in\mathcal B$.  To
analyze $\Lambda(\bm{b})$, $\bm{b}\in\mathcal B$, and find
$\lambda^\ast(\hat{C}_\ell)$, we reparametrize $\bm{b}=(b_1,b_2,b_3)$ via
$x_j=\log(b_j)$, $j=1,2,3$, so that the constraint $b_1b_2b_3=1$
yields $x_1+x_2+x_3=0$.
\rev{
By Proposition~\ref{prop:basic:properties:bast}~\ref{item:Aast:compact}, every maximizer belongs to
$\mathcal B_{\ge \lambda}$. We therefore use the log-parametrization
$\bm b=(e^{x_1},e^{x_2},e^{-(x_1+x_2)})$ and maximize via
\begin{align*}
  \lambda^\ast(\hat C_\ell)
  =
  \max_{(x_1,x_2)\in D_\lambda}
  \Lambda(e^{x_1},e^{x_2}, e^{-(x_1+x_2)};\hat C_\ell),
\end{align*}
where $D_\lambda
  =
  \{
    (x_1,x_2)\in\mathbb R^2:
    x_1,x_2,-x_1-x_2
    \in[\log\lambda,-2\log\lambda]
  \}$ is the compact triangle with vertices $(x,y)=(-L,-L)$, $(-L,2L)$, $(2L,-L)$ with $L=-\log\lambda\geq 0$.
}
The results are summarized in Table~\ref{tab:results}.
\begin{table}[t]
  \centering
  \begin{tabular}{lcccc}
    \hline
   Model & $\lambda$ & $\lambda^\ast$ & $\bm b^\ast$
    & $\xi(\bm b^\ast)$ \\
    \hline
    $\one$-1
    & 0.356 & 0.356 & $(1.000,1.000,1.000)$
    & 0.0000 \\
     $\one$-2
    & 0.233 & 0.372 & $(0.630,0.630,2.520)$
    & 0.4271 \\
     $\one$-3
    & 0.208 & 0.266 & $(1.337,1.337,0.559)$
    & 0.1689 \\[2mm]
    $\two$-1
    & 0.377 & 0.378 & $(0.948,0.948,1.113)$
    & 0.0058 \\
     $\two$-2
    & 0.306 & 0.307 & $(0.956,0.956,1.095)$
    & 0.0041 \\
    \hline
  \end{tabular}
  \caption{TDC $\lambda(\hat C_\ell)$, MTCM
    $\lambda^\ast(\hat C_\ell)$, its maximizing direction
    $\bm b^\ast$, and the associated imbalance index
    $\xi(\bm b^\ast)$ for the five models given
    in~Table~\ref{tab:1:tawn}.}
  \label{tab:results}
\end{table}
We observe different types of behavior among the five models which cannot be
inferred from their log-likelihood values given in Table~\ref{tab:1:tawn}.
Comparing type $\one$ and type $\two$ models, we find that the TDCs of the
estimated type $\two$ models and the symmetric Gumbel model $\one$-1 are
approximately attained by the MTCM at $\bm{b}^\ast=(1,1,1)$, but this is not the
case for the model $\one$-2 (under the constraint $\theta_1=\theta_2=1$)
with MTCM attained at $\bm{b}^\ast=(0.630, 0.630, 2.520)$ and the model
$\one$-3 (under the constraint $\theta_1=\theta_2=\theta$) with MTCM
attained at $\bm{b}^\ast=(1.337, 1.337, 0.559)$.
\rev{The sharp contrast between the maximizing directions of models $\one$-2
and $\one$-3 suggests that parameter restrictions can substantially distort the
inferred optimal tail direction.}
Since the model $\one$-3 has much smaller TDC and MTCM than $\one$-2, the latter
model may overestimate the degree of tail dependence quantified by these measures.

From an insurance perspective, the distinction between $\lambda$ and
$\lambda^\ast$ can be linked to accumulation risk. 
Coastal flood insurance
payouts are driven by rare, high sea levels, and a single storm event can induce
flood losses at multiple locations. 
The dependence structure in the upper
tail of $(Y_1,Y_2,Y_3)$ matters for loss aggregation and capital requirements. 
A model with
$\bm b^\ast\neq(1,1,1)$ indicates that the strongest tail co-movement may occur
along an off-diagonal stress direction, meaning that simultaneous extremes are
most likely when one site is relatively more extreme than the others. This type
of information is relevant for stress testing and scenario design in flood
catastrophe modeling.
\rev{To complement these observations, the next subsection illustrates
more directly how tail non-exchangeability can affect insurance risk assessments.}

\rev{

\subsection{Basis risk under a regional parametric cover}\label{sec:basis:risk}

Parametric (index-based) insurance, often applied to transfer flood and coastal
risks, offers payouts triggered by an objectively measurable physical index
rather than by an assessment of actual losses; see, for example,
\citet{khalil2007enso} and \citet{figueiredo2018probabilistic}.  This type of
insurance enables rapid settlement at relatively low administrative cost, but
gives rise to \emph{basis risk}, namely the risk of a mismatch between the
payout and the actual loss; see, for example, \citet{cummins2004basis}.
Accordingly, evaluating such coverage for basis risk requires careful modeling especially in the
tail; see \citet{kellner2013estimating}. In this section, we consider a regional
multi-site parametric trigger that is activated when the annual maximum sea
levels at each monitored site exceed their
prespecified threshold within the same policy year.

To describe our setting more formally, let $\bm U=(U_1,U_2,U_3)$ follow one of
the five extreme value copulas considered in Section~\ref{sec:model:comparison}.
Let $p=0.01$ and $A_j=\{U_j>1-p\}$ denote the 100-year flood event at site
$j\in\{1,2,3\}$. We consider a cedent portfolio with exposure vector
$\bm \eta=(\eta_1,\eta_2,\eta_3)=(1,1,0.3)$ and a regional parametric cover that pays a
fixed amount $h\in [0,\eta_1+\eta_2+\eta_3]=[0,2.3]$ when all three gauges exceed their
respective 100-year levels. Excluding the contract premium, the cedent's
retained annual loss is
\begin{equation}\label{eq:ret}
L^{\mathrm{ret}}_h
=
\sum_{j=1}^{3}\eta_j\id_{A_j}
-
h\,\id_{A_1\cap A_2\cap A_3}.
\end{equation}
The distribution of $L^{\mathrm{ret}}_h$ is completely determined by
\begin{align*}
q(T)
=
\P\biggl(\,
\bigcap_{j\in T}A_j
\cap
\bigcap_{j\notin T}A_j^c
\biggr),\qquad T\subseteq \{1,2,3\},
\end{align*}
which is, for a given $T\subseteq \{1,2,3\}$, the probability that a 100-year flood occurs
precisely at the sites contained in $T$.
These probabilities can be computed analytically from the underlying copula; see
Table~\ref{tab:pattern}.
\begin{table}[t]
\centering
\small
\begin{tabular}{lccccccc}
\toprule
& \multicolumn{7}{c}{$q(T)\times 10^4$} \\
\cmidrule(lr){2-8}
Model & $\{1\}$ & $\{2\}$ & $\{3\}$ & $\{1,2\}$ & $\{1,3\}$ & $\{2,3\}$ & $\{1,2,3\}$ \\
\midrule
I-1  & 44.35 & 44.35 & 44.35 & \phantom{0}9.87 & 9.87 & 9.87 & 35.91 \\
I-2  & 31.36 & 31.36 & 74.94 & 44.25 & 0.67 & 0.67 & 23.72 \\
I-3  & 27.98 & 27.98 & 73.90 & \textbf{48.35} & 2.43 & 2.43 & \textbf{21.24} \\[2mm]
II-1 & 32.85 & 32.85 & 46.40 & 21.37 & 7.82 & 7.82 & 37.97 \\
II-2 & 25.09 & 25.09 & 57.01 & 37.92 & 5.99 & 5.99 & 31.01 \\
\bottomrule
\end{tabular}
\caption{Model-implied probabilities multiplied by $10^4$ of the non-empty exact
site patterns at the marginal 100-year flood levels.
Boldface highlights the largest probability of the pattern
$\{1,2\}$ and the smallest probability of the pattern $\{1,2,3\}$,
both attained under model I-3.}
\label{tab:pattern}
\end{table}

For this portfolio, the event
$A_{1,2,3^{\mathrm{c}}}=A_1\cap A_2\cap A_3^{\mathrm c}$
is central to the evaluation of severe basis risk.
Indeed, among all non-triggering patterns, $A_{1,2,3^{\mathrm{c}}}$ uniquely produces
the largest retained loss $\eta_1+\eta_2=2$: both high-exposure sites
are flooded, but the absence of an exceedance at the third site
prevents any payout.
This loss cannot be reduced by increasing $h$.
Thus, although basis risk is not confined to $A_{1,2,3^{\mathrm{c}}}$, its probability
$\P(A_{1,2,3^{\mathrm{c}}})=q(\{1,2\})$ is a key diagnostic of severe losses left
uncovered by the regional trigger.
Table~\ref{tab:stats} reports the expected retained loss at the
representative payout $h=1.5$ together with two basis-risk
diagnostics.
\begin{table}[t]
\centering
\small
\setlength{\tabcolsep}{4.5pt}
\begin{tabular}{lccccc}
\toprule
& I-1 & I-2 & I-3 & II-1 & II-2 \\
\midrule
Expected loss $\E[L^{\mathrm{ret}}_{1.5}]$ 
  & 0.0176 & 0.0194 & \textbf{0.0198} & 0.0173 & 0.0183 \\
$\P(A_{1,2,3^{\mathrm{c}}})\times 10^4$
  & 9.87 & 44.25 & \textbf{48.35} & 21.37 & 37.92 \\
Failure rate $\P(A_3^c\mid A_1\cap A_2)$  & 0.216 & 0.651 & \textbf{0.695} & 0.360 & 0.550 \\
\bottomrule
\end{tabular}
\caption{Model-implied expected retained loss at $h=1.5$ and
basis-risk diagnostics for the regional parametric cover with
$p=0.01$ and $\bm{\eta}=(1,1,0.3)$. Boldface indicates the most adverse
value in each row, all attained under Type~I-3.}
\label{tab:stats}
\end{table}

We quantify the tail risk of $L_h^{\mathrm{ret}}$ by expected
shortfall at level $\alpha=0.995$; see
\citet{acerbi2002coherence} and \citet{mcneil2015quantitative}.
Writing $(x)_+=\max\{x,0\}$, its exact expression in terms of
the pattern probabilities is
\begin{align*}
\ES_{\alpha}(L_h^{\mathrm{ret}})
=
\min_{x\in\mathbb R}
\left\{
x+\frac{1}{1-\alpha}
\sum_{T\subseteq\{1,2,3\}}q(T)
\left(
\sum_{j\in T}\eta_j
-h\,\id_{\{T=\{1,2,3\}\}}
-x
\right)_+
\right\}.
\end{align*}
The minimum is attained at an $\alpha$-quantile of
$L_h^{\mathrm{ret}}$ and can therefore be computed by evaluating
the expression at its finitely many retained-loss values.
\begin{figure}[t]
\centering
\includegraphics[width=0.72\textwidth]{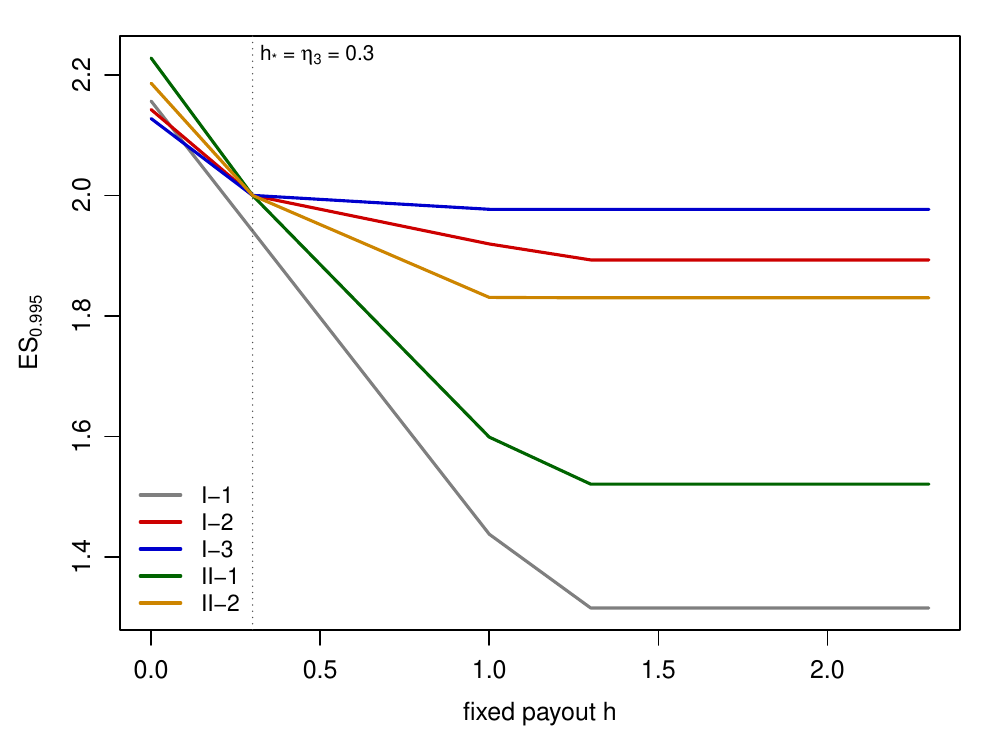}
\caption{The analytically computed $\ES_{0.995}$ of the retained
  annual loss~\eqref{eq:ret} as a function of the fixed payout $h$
  with $\bm{\eta}=(1,1,0.3)$.}
\label{fig:netESh}
\end{figure}
Figure~\ref{fig:netESh} shows $\ES_{0.995}$,  computed analytically from the
pattern probabilities $q(T)$, as a function of the fixed payout $h$.  The main
finding is that the model yielding the most conservative risk assessment changes
as $h$ increases.  At $h=0$, Type~II-1 yields the largest $\ES_{0.995}$.  The
most conservative model then changes to Type~I-3 at the structural change point
$h_{*}=\eta_3=0.3$ indicated by a vertical dotted line in Figure~\ref{fig:netESh}.

As shown in Table~\ref{tab:results}, Type~I-3 has the smallest values
of $\lambda$ and $\lambda^\ast$ among the five fitted models.
Nevertheless, for $h>h_*$, it yields the most conservative risk assessment, suggesting that the other fitted models may
underestimate the tail risk. More broadly, this example
illustrates that a risk assessment depends not only on the magnitude of
tail concordance, but also on the direction in which tail dependence
is strongest, and on how joint tail mass is distributed across
non-exchangeable exceedance patterns. The maximizing direction
$\bm b^\ast$ and related tail non-exchangeability may therefore provide
important information for the design and valuation of a regional
parametric cover contract.
}

\section{Conclusion}\label{sec:concl}

We introduced a multivariate maximal tail concordance measure (MTCM) for
tail-dependent copulas by maximizing the tail copula over rectangles of unit
volume. In contrast to the classical tail dependence coefficient (TDC), which
evaluates extremal dependence only along the diagonal, the proposed measure
accounts for off-diagonal stress directions in the joint tail. The associated
maximizer, when unique, provides an interpretable description of the direction
along which tail dependence is strongest.  We established several basic
properties of the multivariate MTCM. We then derived analytical representations
for important model classes, such as copulas in the maximum domain of attraction
of certain extreme value copulas, Marshall--Olkin copulas, Archimax and nested
Archimedean copulas with regularly varying generators.  \rev{An application to
  trivariate annual sea-level maxima demonstrated the practical relevance of the
  proposed measure. Models with similar likelihood values, or even similar TDC
  values, can nevertheless imply substantially different MTCM values and
  maximizing directions. The example of a regional parametric cover further showed that the fitted model with the smallest TDC and MTCM values yielded the largest retained-loss expected shortfall. 
  These findings
  illustrate the importance of capturing non-exchangeable joint tail features for
  insurance risk assessments.}

\section*{Funding statement}
Takaaki Koike is supported by the Japan Society for the Promotion of Science
(JSPS) KAKENHI grant numbers JP24K00273 and JP26K21178.

\section*{Acknowledgements}
During the preparation of this work, the authors used ChatGPT (GPT-5.6 Sol and GPT-6 Astra; OpenAI) and Claude (Claude Fables 5 and 5.1; Anthropic) to assist in refining ideas and developing mathematical arguments and proofs. The authors reviewed and edited the content as needed and independently verified the mathematical results and proofs. The authors take full responsibility for the content of the published article.


\section*{Competing interests}
The authors declare none.

\section*{Data availability statement}
The code will be made available upon request.


\bibliographystyle{apalike}
\bibliography{biblio.bib}

\appendix

\section*{Appendix}

\section{Proofs}\label{sec:proofs}

This appendix collects all the proofs.  For two functions $f,g:\mathbb{R}\to \mathbb{R}$, we write $f(u)\simeq g(u)$, $u \to u^\ast \in [-\infty,\infty]$, if $\lim_{u \to u^\ast}(f(u)/g(u))=1$.

\subsection{Proposition~\ref{prop:basic:properties:mtcm}}

\begin{proof}[Proof of Proposition~\ref{prop:basic:properties:mtcm}]
  For every tail copula $\Lambda:(0,\infty)^d\to[0,\infty)$, we use the bounds
  \begin{align}
    0\le \Lambda(\bx)\le \min(x_1,\dots,x_d), \quad \bx\in(0,\infty)^d,
    \label{eq:tailcopula:min:bound}
  \end{align}
  and
  \begin{align}
    |\Lambda(\bx)-\Lambda(\bm y)|\le \|\bx-\bm y\|_1,
    \quad \bx,\bm y\in(0,\infty)^d.
    \label{eq:tailcopula:lipschitz:bound}
  \end{align}
  Indeed, if $\Lambda$ is the tail copula of a copula $C$, then
  \eqref{eq:tailcopula:min:bound} follows from
  $0\le C(t\bx)\le \min(tx_1,\dots,tx_d)$ for all sufficiently small $t>0$, while
  \eqref{eq:tailcopula:lipschitz:bound} follows from the $1$-Lipschitz property
  $|C(\bm u)-C(\bm v)|\le \|\bm u-\bm v\|_1$,
  $\bm u,\bm v\in[0,1]^d$, of all copulas.
  We also use that tail copulas are componentwise nondecreasing and $1$-homogeneous.

  \begin{enumerate}[label=(\roman*), labelwidth=\widthof{(iii)}]

  \item
    If $\Lambda\equiv 0$, then $\lambda^\ast(C)=0$ and the supremum is attained at any point.
    Assume now that $\Lambda\not\equiv 0$. Then there exists
    $\bx\in(0,\infty)^d$ such that $\Lambda(\bx;C)>0$.
    Set $r=(\,\prod_{j=1}^d x_j)^{1/d}$ and $\tilde{\bm b}=\bx/r\in\mathcal B$.
    By $1$-homogeneity, we have $\Lambda(\tilde{\bm b};C)=r^{-1}\Lambda(\bx;C)>0$.
    Choose $\eta\in(0,\Lambda(\tilde{\bm b};C))$ and consider the compact set $\mathcal B_{\ge \eta}$ defined in~\eqref{eq:K:c}.

    Next we claim that
    \begin{align}
      \sup_{\bm b\in\mathcal B}\Lambda(\bm b;C)
      =
      \sup_{\bm b\in \mathcal B_{\ge \eta}}\Lambda(\bm b;C).
      \label{eq:claim:compact:reduction:new}
    \end{align}
    Indeed, if $\bm b\in\mathcal B\setminus \mathcal B_{\ge \eta}$, then $\min_j b_j<\eta$. Therefore, by
    \eqref{eq:tailcopula:min:bound}, we have $\Lambda(\bm b;C)\le \min_j b_j<\eta$, and consequently $\sup_{\bm b\in\mathcal B\setminus \mathcal B_{\ge \eta}}\Lambda(\bm b;C)
    \le \eta
    <\Lambda(\tilde{\bm b};C) 
    \le \sup_{\bm b\in\mathcal B}\Lambda(\bm b;C)$.
    This implies \eqref{eq:claim:compact:reduction:new}.

    Finally, $\Lambda(\cdot;C)$ is continuous
    by~\eqref{eq:tailcopula:lipschitz:bound}, and $\mathcal B_{\ge \eta}$ is compact. Hence there
    exists $\bm b^\ast\in \mathcal B_{\ge \eta}$ such that
    $\Lambda(\bm b^\ast;C) = \max_{\bm b\in \mathcal B_{\ge \eta}}\Lambda(\bm b;C) = \sup_{\bm
      b\in\mathcal B}\Lambda(\bm b;C)$.  Thus the supremum
    in~\eqref{eq:def:mtcm:multivariate} is attained and
    $\lambda^\ast(C)=\max_{\bm b\in\mathcal B}\Lambda(\bm b;C)$.

  \item If $\Lambda(\cdot;C)=\Lambda(\cdot;M)$, then
    $\lambda^\ast(C) = \sup_{\bm b\in\mathcal B}\min\{b_1,\dots,b_d\}=1$ since
    $\prod_{j=1}^d b_j=1$ implies $\min\{b_1,\dots,b_d\}\le 1$, and equality is
    attained at $\bm b=\bone_d$.

    Conversely, suppose that $\lambda^\ast(C)=1$. By~\ref{item:attainability:mv},
    there exists $\bm b^\ast\in\mathcal B$ such that $\Lambda(\bm b^\ast;C)=1$.
    Using \eqref{eq:tailcopula:min:bound}, we have $1=\Lambda(\bm b^\ast;C)\le \min\{b_1^\ast,\dots,b_d^\ast\}\le 1$.
    Hence $\min\{b_1^\ast,\dots,b_d^\ast\}=1$. Since $\prod_{j=1}^d b_j^\ast=1$, one must have
    $\bm b^\ast=\bone_d$ and thus $\Lambda(\bone_d;C)=1$.

    Now fix $\bx\in(0,\infty)^d$ and put $m:=\min\{x_1,\dots,x_d\}$. Since
    $m\bone_d\le \bx$ componentwise, monotonicity and $1$-homogeneity imply that
    $\Lambda(\bx;C)\ge \Lambda(m\bone_d;C)=m\,\Lambda(\bone_d;C)=m$.  Together
    with \eqref{eq:tailcopula:min:bound}, this yields
    $m\le \Lambda(\bx;C)\le m$, so $\Lambda(\bx;C)=m=\Lambda(\bx;M)$,
    $\bx\in(0,\infty)^d$.  Hence $\Lambda(\cdot;C)=\Lambda(\cdot;M)$.

  \item
    If $\Lambda(\cdot;C)\equiv 0$, then clearly $\lambda^\ast(C)=0$.
    Conversely, suppose that $\lambda^\ast(C)=0$. Then $\Lambda(\bm b;C)=0$, $\bm b\in\mathcal B$.
    Let $\bx\in(0,\infty)^d$ and set $r=(\,\prod_{j=1}^d x_j)^{1/d}$ and $\bm b=\bx/r\in\mathcal B$.
    By $1$-homogeneity, we have $\Lambda(\bx;C)=\Lambda(r\bm b;C)=r\,\Lambda(\bm b;C)=0$.
    Hence
    $\Lambda(\cdot;C)\equiv 0$ on $(0,\infty)^d$.

  \item
    If $\Lambda_1(\bx;C_1)\le \Lambda_2(\bx;C_2)$ for all
    $\bx\in(0,\infty)^d$, then
    $\lambda^\ast(C_1) = \sup_{\bm b\in\mathcal B}\Lambda_1(\bm b;C_1) \le \sup_{\bm b\in\mathcal B}\Lambda_2(\bm b;C_2)$ $= \lambda^\ast(C_2)$.

  \item
    Let $t\in[0,1]$. Since the set of all copulas is convex, $tC_1+(1-t)C_2$ is again
    a copula. Moreover,
    \begin{align*}
      \Lambda(\bx;tC_1+(1-t)C_2)
      =
      \lim_{u\downarrow 0}\frac{tC_1(u\bx)+(1-t)C_2(u\bx)}{u}
      =
      t\Lambda_1(\bx;C_1)+(1-t)\Lambda_2(\bx;C_2).
    \end{align*}
    Therefore,
    \begin{align*}
      \lambda^\ast(tC_1+(1-t)C_2)
      &=
        \sup_{\bm b\in\mathcal B}
        \{t\Lambda_1(\bm b;C_1)+(1-t)\Lambda_2(\bm b;C_2)\}\\
      &\le
        t\sup_{\bm b\in\mathcal B}\Lambda_1(\bm b;C_1)
        +(1-t)\sup_{\bm b\in\mathcal B}\Lambda_2(\bm b;C_2)\\
      &=
        t\lambda^\ast(C_1)+(1-t)\lambda^\ast(C_2).
    \end{align*}

  \item
    Let $\Lambda_n(\cdot):=\Lambda(\cdot;C_n)$ and
    $\Lambda(\cdot):=\Lambda(\cdot;C)$.
    We first show that the convergence $\Lambda_n\to\Lambda$ is uniform
    on every compact $K\subset(0,\infty)^d$. Fix any $\varepsilon>0$.
    Then there exist
    $\bm z^{(1)},\dots,\bm z^{(m)}\in K$ such that $K\subset \bigcup_{k=1}^m B_1(\bm z^{(k)},\varepsilon/3)$, where $B_1(\bm{z}^{(k)},\varepsilon/3)$ denotes the ball in 1-norm with radius
    $\varepsilon/3$ centered at $\bm{z}^{(k)}$. By pointwise convergence at these
    finitely many points, there exists $n_0\in\mathbb N$ such that
    $|\Lambda_n(\bm z^{(k)})-\Lambda(\bm z^{(k)})|<\varepsilon/3$, $k=1,\dots,m$
    for all $n\ge n_0$.  Fix $n\ge n_0$ and $\bm z\in K$. Choose $k$ such that
    $\|\bm z-\bm z^{(k)}\|_1<\varepsilon/3$. Then
    \begin{align*}
      |\Lambda_n(\bm z)-\Lambda(\bm z)|
      \le
      |\Lambda_n(\bm z)-\Lambda_n(\bm z^{(k)})|
      +|\Lambda_n(\bm z^{(k)})-\Lambda(\bm z^{(k)})|
      +|\Lambda(\bm z^{(k)})-\Lambda(\bm z)|
      <\varepsilon,
    \end{align*}
    where we used \eqref{eq:tailcopula:lipschitz:bound} for the first and third
    terms. Hence $\sup_{\bm z\in K}|\Lambda_n(\bm z)-\Lambda(\bm z)|\to 0$.

    Assume first that $\lambda^\ast(C)>0$. Choose
    $\tilde{\bm b}\in\mathcal B$ such that $\Lambda(\tilde{\bm b})>0$, and fix
    $\eta\in(0,\Lambda(\tilde{\bm b}))$. Define the compact set $\mathcal B_{\ge \eta}$ by~\eqref{eq:K:c}.
    By~\eqref{eq:tailcopula:min:bound}, it follows that $\sup_{\bm b\in\mathcal B\setminus\mathcal B_{\ge\eta}} \Lambda_n(\bm b)\le\eta$ and $\sup_{\bm b\in\mathcal B\setminus \mathcal B_{\ge \eta}}\Lambda(\bm b)\le \eta$.
    Since $\Lambda_n(\tilde{\bm b})\to \Lambda(\tilde{\bm b})>\eta$, there exists
    $n_1\in\mathbb N$ such that $\Lambda_n(\tilde{\bm b})>\eta$ for $n\ge n_1$.
    Hence, for $n\ge n_1$, we have $\lambda^\ast(C_n)=\sup_{\bm b\in \mathcal B_{\ge \eta}}\Lambda_n(\bm b)$ and $
    \lambda^\ast(C)=\sup_{\bm b\in \mathcal B_{\ge \eta}}\Lambda(\bm b)$ by~\eqref{eq:claim:compact:reduction:new}.
    Since $\Lambda_n\to\Lambda$ uniformly on $\mathcal B_{\ge \eta}$, we have
    \begin{align*}
      |\lambda^\ast(C_n)-\lambda^\ast(C)|
      \le
      \sup_{\bm b\in \mathcal B_{\ge \eta}}|\Lambda_n(\bm b)-\Lambda(\bm b)|
      \to 0,
    \end{align*}
    and thus $\lambda^\ast(C_n)\to\lambda^\ast(C)$.

    Finally, assume $\lambda^\ast(C)=0$.
    Then $\Lambda\equiv 0$ by~\ref{item:zero:mv}. Fix $\eta\in(0,1)$ and define $\mathcal B_{\ge \eta}$ as in~\eqref{eq:K:c}.
    For
    every $n\in\mathbb N$, since $\Lambda_n(\bm b)\le \eta$ on $\mathcal B\setminus \mathcal B_{\ge \eta}$
    by~\eqref{eq:tailcopula:min:bound}, we have
    \begin{align*}
      \lambda^\ast(C_n)
      =
      \sup_{\bm b\in\mathcal B}\Lambda_n(\bm b)
      \le
      \max\bigl\{
      \sup_{\bm b\in \mathcal B_{\ge \eta}}\Lambda_n(\bm b),\,\eta
      \bigr\}.
    \end{align*}
    Since $\Lambda_n\to 0$ uniformly on
    $\mathcal B_{\ge \eta}$, we obtain $\sup_{\bm b\in \mathcal B_{\ge \eta}}\Lambda_n(\bm b)\to 0$.  Therefore,
    $\limsup_{n\to\infty}\lambda^\ast(C_n)\le \eta$.  As $\eta\in(0,1)$ was arbitrary
    and $\lambda^\ast(C_n)\ge 0$, we conclude that
    $\lambda^\ast(C_n)\to 0=\lambda^\ast(C)$.  \qedhere
  \end{enumerate}
\end{proof}

\rev{
\subsection{Proposition~\ref{prop:basic:properties:bast}}

\begin{proof}[Proof of Proposition~\ref{prop:basic:properties:bast}]
\leavevmode
  \begin{enumerate}[label=(\roman*), labelwidth=\widthof{(iii)}]
  \item
    Let $\pi\in\mathcal S$. Since $\bm b_\pi^\ast\in\mathcal B$ and $\Lambda(\bm b_\pi^\ast;C)
      =
      \Lambda(\bm b^\ast;C)
      =
      \lambda^\ast(C)$ by assumption, we have  $\bm b_\pi^\ast\in\mathcal B^\ast(C)$. Since the maximizer is
    unique, it follows that $\bm b_\pi^\ast=\bm b^\ast$.

  \item
    Let $\bm b^\ast\in\mathcal B^\ast(C)$.
    Since $\lambda^\ast(C)
      =
      \Lambda(\bm b^\ast;C)
      \leq
      \min_{1\leq j\leq d}b_j^\ast$, we have $\mathcal B^\ast(C)\subseteq\mathcal B_{\ge \lambda^\ast(C)}$.
      The second relationship  $\mathcal B_{\ge \lambda^\ast(C)}
    \subseteq\mathcal B_{\ge \lambda(C)}$ follows from $\lambda(C)\leq\lambda^\ast(C)$.
Finally, compactness of $\mathcal B^\ast(C)$ follows from continuity of $\Lambda(\cdot;C)$ and
  $\mathcal B^\ast(C)
      =
      \mathcal B_{\ge \lambda^\ast(C)}
      \cap
      \{\bm b\in\mathcal B:
        \Lambda(\bm b;C)=\lambda^\ast(C)
      \}$.

  \item
  By~\ref{item:Aast:compact}, we have $b_j^\ast\ge\lambda^\ast(C)$ for every
$j\in\{1,\dots,d\}$. Multiplication over $I$ and its complement, together with
$\prod_{j=1}^d b_j^\ast=1$, gives the product bounds, while the sum bound is
immediate. The remaining inequalities follow from
$\lambda(C)\le\lambda^\ast(C)$.
  \item

  The inclusion
$\subseteq$
follows immediately from~\ref{item:Aast:compact}.
    Conversely, fix $\bm b\in\mathcal B_{\ge \eta}$ and define $\alpha_j
      =
 \eta/b_j$, $j=1,\ldots,d$.
    Then $\alpha_j\in(0,1]$ and $ \prod_{j=1}^d\alpha_j
      =
      \eta^d$.
As shown in Proposition~\ref{prop:mo:mtcm}, the respective survival Marshall--Olkin copula with this parameter vector $(\alpha_1,\dots,\alpha_d)$ has the MTCM $\eta$ and is uniquely attained.
Consequently, we have the converse inclusion and complete the proof.
    Note that the proof of Proposition~\ref{prop:mo:mtcm} is
independent of the present proposition; hence, this forward reference
does not introduce any circularity.

  \item
    Write the constraint $\prod_{j=1}^d b_j=1$ as $h(\bm b)
      =
      \sum_{j=1}^d\log b_j
      =
      0$.
    Since $\nabla h(\bm b^\ast)\neq\bzero_d$, the Lagrange multiplier
    theorem gives a constant $\mu\in \R$ such that
    $
      (\partial/\partial x_j)\Lambda(\bm x;C) |_{\bm{x} = \bm{b}^\ast}
      =
      \mu/b_j^\ast$, $j=1,\ldots,d$.
        Differentiating the homogeneity identity
    $\Lambda(t\bm b^\ast;C)=t\Lambda(\bm b^\ast;C)$ at $t=1$ gives
    \begin{align*}
      \lambda^\ast(C)
      =
      \Lambda(\bm b^\ast;C)
      =
      \sum_{j=1}^d
      b_j^\ast \frac{\partial}{\partial x_j}\Lambda(\bm x;C) \Big|_{\bm{x} = \bm{b}^\ast}
      =
      d\mu,
    \end{align*}
    which proves the desired result.

  \item
Write $\lambda_n^\ast=\lambda^\ast(C_n)$ and
$\lambda^\ast=\lambda^\ast(C)$. The uniqueness assumption and
Proposition~\ref{prop:basic:properties:mtcm}~\ref{item:zero:mv}
imply $\lambda^\ast>0$. Moreover,
Proposition~\ref{prop:basic:properties:mtcm}~\ref{item:continuity:mv}
gives $\lambda_n^\ast\to\lambda^\ast$.
Set $\eta=\lambda^\ast/2$. For all sufficiently large $n$,
we have $\lambda_n^\ast\ge \eta$, and hence~\ref{item:Aast:compact} yields $\mathcal B^\ast(C_n)
  \subseteq
  \mathcal B_{ \ge \lambda_n^\ast}
  \subseteq
  \mathcal B_{\ge \eta}$.
Thus every sequence $\bm b_n^\ast\in\mathcal B^\ast(C_n)$ is eventually
contained in the compact set $\mathcal B_{\ge \eta}$.
Let $\bm b_{n_k}^\ast\to{\bm b}^\circ$ along an arbitrary convergent
subsequence. Uniform convergence of $\Lambda_n$ to $\Lambda$ on
$\mathcal B_{\ge \eta}$ gives $\Lambda({\bm b}^\circ)
  =
  \lim_{k\to\infty}\Lambda_{n_k}(\bm b_{n_k}^\ast)
  =
  \lim_{k\to\infty}\lambda_{n_k}^\ast
  =
  \lambda^\ast$.
Therefore
${\bm b}^\circ\in\mathcal B^\ast(C)=\{\bm b^\ast\}$.
Every subsequential limit is thus $\bm b^\ast$, and consequently
$\bm b_n^\ast\to\bm b^\ast$.
    \qedhere
  \end{enumerate}
\end{proof}
}

\subsection{Proposition~\ref{prop:mo:mtcm}}
\begin{proof}[Proof of Proposition~\ref{prop:mo:mtcm}]
  \mbox{}
  \begin{enumerate}[label=(\roman*), labelwidth=\widthof{(iii)}]
  \item
   For each $\emptyset\neq S\subseteq\{1,\dots,d\}$, we have that
    \begin{align*}
      \ell_{\bm\alpha,S}(\bx)
      =
      \lim_{x_j\downarrow 0,\ j\notin S}\ell_{\bm\alpha}(\bx)
      =
      \sum_{j\in S}(1-\alpha_j)x_j+\max_{j\in S}\{\alpha_j x_j\}.
    \end{align*}
    Substituting this into~\eqref{eq:tail:copula:stdf}, we obtain
    \begin{align}
      \Lambda(\bx;\hat C^{\mathrm{MO}}_{\bm\alpha})
      &=
        \sum_{\emptyset\neq S\subseteq\{1,\dots,d\}}
        (-1)^{|S|-1}
        \sum_{j\in S}(1-\alpha_j)x_j
        +
        \sum_{\emptyset\neq S\subseteq\{1,\dots,d\}}
        (-1)^{|S|-1}\max_{j\in S}\{\alpha_j x_j\}.
        \label{eq:mo:tailcopula:split}
    \end{align}
    The first term in~\eqref{eq:mo:tailcopula:split} vanishes since
    \begin{align*}
      \sum_{\emptyset\neq S\subseteq\{1,\dots,d\}}
      (-1)^{|S|-1}\sum_{j\in S}(1-\alpha_j)x_j
      &=
        \sum_{j=1}^d (1-\alpha_j)x_j
        \sum_{S: j\in S}(-1)^{|S|-1},
    \end{align*}
    and $\sum_{S:j\in S}(-1)^{|S|-1}=(1-1)^{d-1}=0$ for each fixed $j$.

    For the second term in~\eqref{eq:mo:tailcopula:split}, put
$y_j=\alpha_j x_j\ge 0$ and let $y_{(1)}\le\cdots\le y_{(d)}$ denote the
order statistics of $y_1,\dots,y_d$ with ties broken arbitrarily.  Grouping subsets $S$ according to the
largest rank contained in $S$, we obtain
\begin{align*}
  \Lambda(\bx;\hat C^{\mathrm{MO}}_{\bm\alpha})=\sum_{\emptyset\neq S\subseteq\{1,\dots,d\}} (-1)^{|S|-1}\max_{j\in S}\{y_j\}
&=\sum_{k=1}^d y_{(k)} \sum_{s=1}^{k}\binom{k-1}{s-1}(-1)^{s-1}. 
\end{align*}
Using that
\begin{align*}
  \sum_{s=1}^{k}\binom{k-1}{s-1}(-1)^{s-1}=\sum_{s=0}^{k-1}\binom{k-1}{s}(-1)^s=(1-1)^{k-1}=\mathbbm{1}_{\{k=1\}},
\end{align*}
we obtain $\Lambda(\bx;\hat C^{\mathrm{MO}}_{\bm\alpha})=y_{(1)}=\min\{y_1,\dots,y_d\}$, which is the desired formula.
  \item
    Next, we study the MTCM
    $\lambda^\ast(\hat C^{\mathrm{MO}}_{\bm\alpha})
    =
    \sup_{\bm b\in\mathcal B}\min_{1\le j\le d}\{\alpha_j b_j\}$.
    For every $\bm b\in\mathcal B$, we have $(\min_{1\le j\le d}\{\alpha_j b_j\})^d\le \prod_{j=1}^d \alpha_j b_j$
    and thus
    \begin{align}
      \min_{1\le j\le d}\{\alpha_j b_j\}
      \le
      \biggl(\,\prod_{j=1}^d \alpha_j b_j\biggr)^{1/d}
      =\biggl(\,\prod_{j=1}^d \alpha_j\prod_{j=1}^d b_j\biggr)^{1/d}
      =\prod_{j=1}^d \alpha_j^{1/d}.
      \label{eq:mo:agm:bound}
    \end{align}
    Hence $\lambda^\ast(\hat C^{\mathrm{MO}}_{\bm\alpha})
    \le
    \prod_{j=1}^d \alpha_j^{1/d}$.
    On the other hand, let $\bar\alpha=\prod_{j=1}^d \alpha_j^{1/d}$ and $b_j^\ast=\bar\alpha/\alpha_j$, $j=1,\dots,d$.
    Then $ \prod_{j=1}^d b_j^\ast=1$ and thus  $\bm b^\ast\in\mathcal B$.
    Moreover, $\min_{1\le j\le d}\{\alpha_j b_j^\ast\}
    =
    \min_{1\le j\le d}\{\bar\alpha\}
    =
    \bar\alpha$, which implies $\lambda^\ast(\hat C^{\mathrm{MO}}_{\bm\alpha})
    \ge \bar\alpha$.
    Consequently, we have $\lambda^\ast(\hat C^{\mathrm{MO}}_{\bm\alpha})
    = \bar\alpha$.

  \item In~\eqref{eq:mo:agm:bound}, equality holds if and only if $\alpha_1b_1=\cdots=\alpha_db_d$.
    Hence any maximizer $\bm b\in\mathcal B$ must satisfy $b_j=c/\alpha_j$, $j=1,\dots,d$, for some constant $c>0$.  Using the constraint $\prod_{j=1}^d b_j=1$, the constant is necessarily  $c=\prod_{j=1}^d\alpha_j^{1/d}$, which yields the desired formula.\qedhere
  \end{enumerate}
\end{proof}

\subsection{Proposition~\ref{prop:tail:cop:AC:reg:var:gen}}

For the proof of Proposition~\ref{prop:tail:cop:AC:reg:var:gen}, we need the
following notion.  A measurable function $L:(0,\infty)\to (0,\infty)$ is \emph{slowly
  varying} (at $\infty$) if $\lim_{x\to\infty} L(tx)/L(x)=1$ for all $t>0$.
Let $\mathrm{SV}$ denote the class of all slowly varying functions.
According to \citet[Theorem~1.4.1]{binghamgoldieteugels1987}, any function $f \in\mathrm{RV}_\rho$, $\rho \in \R$, can be represented by $f(x)=x^{\rho}L(x)$ for some $L\in \mathrm{SV}$.

\begin{proof}[Proof of Proposition~\ref{prop:tail:cop:AC:reg:var:gen}]
  Let $\psi\in\Psi_\infty\cap \mathrm{RV}_{-\alpha}$ for some $\alpha>0$.  Then
  there exists $L\in \mathrm{SV}$ such that $\psi(t)=t^{-\alpha}L(t)$.  By \citet[Theorem~1.5.12]{binghamgoldieteugels1987}, we can also write $\psi^{-1}(u)=u^{-1/\alpha}\tilde{L}(1/u)$ for some
  $\tilde{L}\in\mathrm{SV}$.

  The two functions $L,\tilde{L}$ are related via
  \begin{align}
    \tilde{L}\biggl(\frac{1}{u}\biggr)^{-\alpha}L\biggl(u^{-1/\alpha}\tilde{L}\biggl(\frac{1}{u}\biggr)\biggr)=1\label{eq:asym:L:and:tilde:L}
  \end{align}
  since, for every $u\in(0,1)$,
  \begin{align*}
    u&=\psi(\psi^{-1}(u))=\psi^{-1}(u)^{-\alpha}L(\psi^{-1}(u))=\biggl\{u^{-1/\alpha}\tilde{L}\biggl(\frac{1}{u}\biggr)\biggr\}^{-\alpha}L\biggl(u^{-1/\alpha}\tilde{L}\biggl(\frac{1}{u}\biggr)\biggr)\\
                      &=u\tilde{L}\biggl(\frac{1}{u}\biggr)^{-\alpha}L\biggl(u^{-1/\alpha}\tilde{L}\biggl(\frac{1}{u}\biggr)\biggr).
  \end{align*}
  Together with $1$-homogeneity of $\ell$, we have
  from~\eqref{eq:asym:L:and:tilde:L} that, for any fixed $\bx \in(0,\infty)^d$ and as $t\downarrow 0$,
  {\allowdisplaybreaks%
  \begin{align*}
    C(t\bx)&=\psi\bigl(\ell(\psi^{-1}(tx_1),\dots,\psi^{-1}(tx_d))\bigr)\\
              &= \bigl\{\ell(\psi^{-1}(tx_1),\dots,\psi^{-1}(tx_d))\bigr\}^{-\alpha}L\bigl(\ell(\psi^{-1}(tx_1),\dots,\psi^{-1}(tx_d))\bigr)\\
              &= \biggl\{\ell\biggl((tx_1)^{-1/\alpha}\tilde{L}\biggl(\frac{1}{tx_1}\biggr),\dots,(tx_d)^{-1/\alpha}\tilde{L}\biggl(\frac{1}{tx_d}\biggr)\biggr)\biggr\}^{-\alpha}\\
              &\phantom{{}={}}\cdot L\biggl(\ell\biggl( (tx_1)^{-1/\alpha}\tilde{L}\biggl(\frac{1}{tx_1}\biggr),\dots,(tx_d)^{-1/\alpha}\tilde{L}\biggl(\frac{1}{tx_d}\biggr) \biggr)\biggr)\\
           &\simeq \biggl\{\ell\biggl((tx_1)^{-1/\alpha}\tilde{L}\biggl(\frac{1}{t}\biggr),\dots,(tx_d)^{-1/\alpha}\tilde{L}\biggl(\frac{1}{t}\biggr)\biggr)\biggr\}^{-\alpha}\\
                &\phantom{{}={}}\cdot L\biggl(\ell\biggl( (tx_1)^{-1/\alpha}\tilde{L}\biggl(\frac{1}{t}\biggr),\dots,(tx_d)^{-1/\alpha}\tilde{L}\biggl(\frac{1}{t}\biggr) \biggr)\biggr)\\
              &= \biggl\{t^{-1/\alpha}\tilde{L}\biggl(\frac{1}{t}\biggr)\ell\bigl(x_1^{-1/\alpha},\dots,x_d^{-1/\alpha}\bigr)\biggr\}^{-\alpha} L\biggl(t^{-1/\alpha}\tilde{L}\biggl(\frac{1}{t}\biggr)\ell\bigl( x_1^{-1/\alpha},\dots,x_d^{-1/\alpha} \bigr)\biggr)\\
              &= t \ell\bigl(x_1^{-1/\alpha},\dots,x_d^{-1/\alpha}\bigr)^{-\alpha} \tilde{L}\biggl(\frac{1}{t}\biggr)^{-\alpha} L\biggl(t^{-1/\alpha}\tilde{L}\biggl(\frac{1}{t}\biggr)\ell\bigl( x_1^{-1/\alpha},\dots,x_d^{-1/\alpha} \bigr)\biggr)\\
              &\simeq t \ell\bigl(x_1^{-1/\alpha},\dots,x_d^{-1/\alpha}\bigr)^{-\alpha} \tilde{L}\biggl(\frac{1}{t}\biggr)^{-\alpha} L\biggl(t^{-1/\alpha}\tilde{L}\biggl(\frac{1}{t}\biggr)\biggr)\\
              &= t \ell\bigl(x_1^{-1/\alpha},\dots,x_d^{-1/\alpha}\bigr)^{-\alpha}.
  \end{align*}}
  Note that the first asymptotic relation $\simeq$ above follows from~\citet[Proposition~2.6~(iii)]{resnick2007heavy}.
  \qedhere
\end{proof}

\subsection{Theorem~\ref{thm:mtcm:archimax:minimization}}

\begin{proof}[Proof of Theorem~\ref{thm:mtcm:archimax:minimization}]
  Since $\prod_{j=1}^d b_j = 1$ if and only if $\prod_{j=1}^d b_j^{-1/\alpha}=1$, the map $T_\alpha:\mathcal B\to\mathcal B$ defined by $T_\alpha(\bm b)=(b_1^{-1/\alpha},\dots,b_d^{-1/\alpha})$ is a bijection with inverse $T_\alpha^{-1}(\bm z)=(z_1^{-\alpha},\dots,z_d^{-\alpha})$.
  By Proposition~\ref{prop:basic:properties:mtcm}, the maximum in the
  definition of $\lambda^\ast(C_{\psi,\ell})$ is attained.
  By Proposition~\ref{prop:tail:cop:AC:reg:var:gen}, we have that
  \begin{align*}
    \lambda^\ast(C_{\psi,\ell})
    = \max_{\bm b\in\mathcal B}\Lambda(\bm b;C_{\psi,\ell})
    = \max_{\bm z\in\mathcal B}\{\ell(\bm z)^{-\alpha}\}
    = \{\min_{\bm z\in\mathcal B}\ell(\bm z)\}^{-\alpha}.
  \end{align*}
  The correspondence between maximizer and minimizer follows immediately.\qedhere
\end{proof}

\rev{
\subsection{Corollary~\ref{cor:mtcm:archimax:symmetric}}

\begin{proof}[Proof of Corollary~\ref{cor:mtcm:archimax:symmetric}]
  By Theorem~\ref{thm:mtcm:archimax:minimization}, it suffices to minimize
  $\ell$ on $\mathcal B$. Let $\bm z=(z_1,\dots,z_d)\in\mathcal B$ and set
  $\bar z=d^{-1}\sum_{j=1}^d z_j$. Since $\ell$ is exchangeable and convex,
  \begin{align*}
    \ell(\bar z\bone_d)
    =
    \ell\biggl(
      \frac{1}{d!}\sum_{\pi\in\mathfrak S_d}\bm z_\pi
    \biggr)
    \le
    \frac{1}{d!}\sum_{\pi\in\mathfrak S_d}\ell(\bm z_\pi)
    =
    \ell(\bm z).
  \end{align*}
  Moreover, since $\prod_{j=1}^d z_j=1$, the arithmetic--geometric mean
  inequality gives $\bar z\geq1$, with equality if and only if
  $\bm z=\bone_d$. Therefore, by the $1$-homogeneity of $\ell$, we have $\ell(\bm z)
    \geq
    \ell(\bar z\bone_d)
    =
    \bar z\,\ell(\bone_d)
    \geq
    \ell(\bone_d)$,
  where the last inequality is strict unless $\bm z=\bone_d$.
  Thus, $\bone_d$ is the unique minimizer of $\ell$ on $\mathcal B$.
  Theorem~\ref{thm:mtcm:archimax:minimization} now yields $\lambda^\ast(C_{\psi,\ell})
    =
    \ell(\bone_d)^{-\alpha}$ and $
    \bm b^\ast
    =
    T_\alpha^{-1}(\bone_d)
    =
    \bone_d$.
\end{proof}
}

\subsection{Proposition~\ref{prop:tail:cop:tree:nested:archimedean}}
\begin{proof}[Proof of Proposition~\ref{prop:tail:cop:tree:nested:archimedean}]
  We proceed by induction on the height of the subtree rooted at $v$.
  If $v \in \mathcal L$, then~\eqref{eq:recursion:tail:copula:tree} is obvious.

  Next, let $v\in\mathcal I$ and assume that the statement holds for every child
  $w\in\mathrm{ch}(v)$.  By the induction hypothesis, we have $C_w(t\bx_{\mathrm{le}(w)}) \simeq t\,\Lambda_w(\bx_{\mathrm{le}(w)})$, $t\downarrow 0$, for every $w\in\mathrm{ch}(v)$.
  Write $\psi_v(t)=t^{-\alpha_v}L_v(t)$ and $\psi_v^{-1}(u)=u^{-1/\alpha_v}\tilde{L}_v(1/u)$ for some $L_v,\tilde{L}_v\in\mathrm{SV}$.
  Then
  \begin{align*}
    \psi_v^{-1}(C_w(t\bx_{\mathrm{le}(w)}))
    \simeq \{t\Lambda_w(\bx_{\mathrm{le}(w)})\}^{-1/\alpha_v}\tilde L_v\biggl(\frac{1}{t\,\Lambda_w(\bx_{\mathrm{le}(w)})}\biggr)
    \simeq \{t\Lambda_w(\bx_{\mathrm{le}(w)})\}^{-1/\alpha_v}\tilde L_v\biggl(\frac{1}{t}\biggr) 
  \end{align*}
  as $t\downarrow 0$, and thus
  \begin{align*}
    \sum_{w\in\mathrm{ch}(v)}
    \psi_v^{-1}(C_w(t\bx_{\mathrm{le}(w)}))
    \simeq
    t^{-1/\alpha_v}\tilde L_v\biggl(\frac{1}{t}\biggr)
    \sum_{w\in\mathrm{ch}(v)}\Lambda_w(\bx_{\mathrm{le}(w)})^{-1/\alpha_v}.
  \end{align*}
  Together with $\tilde{L}_v(1/u)^{-\alpha_v}L_v(u^{-1/\alpha_v}\tilde{L}_v(1/u))= 1$ as in~\eqref{eq:asym:L:and:tilde:L}, we have that
  \begin{align*}
    C_v(t\bx_{\mathrm{le}(v)})
    &= \psi_v\biggl(\,
    \sum_{w\in\mathrm{ch}(v)}
    \psi_v^{-1}(C_w(t\bx_{\mathrm{le}(w)}))\biggr)\\
    &\simeq
    \biggl\{t^{-1/\alpha_v}\tilde{L}_v\biggl(\frac{1}{t}\biggr)\sum_{w\in\mathrm{ch}(v)}\Lambda_w(\bx_{\mathrm{le}(w)})^{-1/\alpha_v}\biggr\}^{-\alpha_v}\!\!\!\!L_v\biggl(t^{-1/\alpha_v}\tilde{L}_v\biggl(\frac{1}{t}\biggr)\sum_{w\in\mathrm{ch}(v)}\Lambda_w(\bx_{\mathrm{le}(w)})^{-1/\alpha_v}\biggr)\\
    &\simeq t\biggl\{\,\sum_{w\in\mathrm{ch}(v)} \Lambda_w(\bx_{\mathrm{le}(w)})^{-1/\alpha_v}\biggr\}^{-\alpha_v}\tilde{L}_v\biggl(\frac{1}{t}\biggr)^{-\alpha_v}L_v\biggl(t^{-1/\alpha_v}\tilde{L}_v\biggl(\frac{1}{t}\biggr)\biggr)\\ 
    &=t\biggl(\,\sum_{w\in\mathrm{ch}(v)} \Lambda_w(\bx_{\mathrm{le}(w)})^{-1/\alpha_v}\biggr)^{-\alpha_v}.
  \end{align*}
  Hence $\lim_{t\downarrow 0}
  C_v(t\bx_{\mathrm{le}(v)})/t
  =
  \Lambda_v(\bx_{\mathrm{le}(v)})$, completing the induction.
\end{proof}

\rev{
\subsection{Theorem~\ref{thm:mtcm:nested:archimedean:recursive}}

\begin{proof}[Proof of Theorem~\ref{thm:mtcm:nested:archimedean:recursive}]
  \mbox{}
  \begin{enumerate}[label=(\roman*), labelwidth=\widthof{(iii)}]
  \item
    For $v\in \mathcal V$ and $p>0$, define
    \begin{align*}
      A_v(p)
      =
      \biggl\{
      \bm b_{\mathrm{le}(v)}\in(0,\infty)^{d(v)}:
      \prod_{i\in\mathrm{le}(v)}b_i=p
      \biggr\}
    \end{align*}
    and
    $M_v(p) = \sup_{\bm b_{\mathrm{le}(v)}\in A_v(p)} \Lambda_v(\bm
    b_{\mathrm{le}(v)})$. 
    Thus $M_v(1)=\lambda^\ast(C_v)$ for $v\in\mathcal I$. By induction on the height of the subtree
    rooted at $v$, we first prove that $M_v(p)=\lambda_v^\ast\,p^{1/d(v)}$
    and that this supremum is uniquely attained in $A_v(p)$ for every $p>0$.

    In the base case, since $d(v)=1$, $A_v(p)=\{p\}$ and $\Lambda_v(b_v)=b_v$,
    the statement holds for $v\in\mathcal L$.  Now consider
    the induction step. Let $v\in\mathcal I$, and assume that the statements
    hold for every child $w\in\mathrm{ch}(v)$. Fix $p>0$. For $x_w>0$,
    $w\in\mathrm{ch}(v)$, let
    \begin{align*}
      S_{\alpha_v}((x_w)_{w\in\mathrm{ch}(v)})
      =
      \biggl(\,
      \sum_{w\in\mathrm{ch}(v)}x_w^{-1/\alpha_v}
      \biggr)^{-\alpha_v}.
    \end{align*}
    Since $\alpha_v>0$, the map $S_{\alpha_v}$ is continuous and strictly
    increasing in each coordinate.
    By Proposition~\ref{prop:tail:cop:tree:nested:archimedean}, for every
    $\bm b_{\mathrm{le}(v)}\in A_v(p)$, we have
    $\Lambda_v(\bm b_{\mathrm{le}(v)})=
    S_{\alpha_v}
    \bigl(
      (
        \Lambda_w(\bm b_{\mathrm{le}(w)})
      )_{w\in\mathrm{ch}(v)}\bigr).
    $
    For $\bm b_{\mathrm{le}(v)}\in A_v(p)$, let
    $p_w
      =
      \prod_{i\in\mathrm{le}(w)} b_i$, $w\in\mathrm{ch}(v)$.
    Then $p_w>0$ 
    and
    $
    \prod_{w\in\mathrm{ch}(v)}p_w=p.
    $
    Conversely, any collection of $p_w>0$, $w\in\mathrm{ch}(v)$, satisfying
    $\prod_{w\in\mathrm{ch}(v)} p_w=p$, together with vectors
    $\bm b_{\mathrm{le}(w)}\in A_w(p_w)$, determines an element of $A_v(p)$.
    Hence
    \begin{align}\label{eq:Mv:p:decomposed}
        M_v(p)
        &=
          \sup_{\prod_{w\in\mathrm{ch}(v)} p_w=p}\bigl\{
          \sup_{\bm b_{\mathrm{le}(w)}\in A_w(p_w),\,w\in\mathrm{ch}(v)}
          S_{\alpha_v}
          \bigl(
          (
          \Lambda_w(\bm b_{\mathrm{le}(w)})
          )_{w\in\mathrm{ch}(v)}
          \bigr)\bigr\}.
    \end{align}
    By the induction hypothesis, $M_w(p_w)$ is attained for every child
    $w\in\mathrm{ch}(v)$ and every $p_w>0$. Therefore, using the componentwise monotonicity
    of $S_{\alpha_v}$, the inner supremum in~\eqref{eq:Mv:p:decomposed} equals
    $
    S_{\alpha_v}
    \bigl(
      (M_w(p_w))_{w\in\mathrm{ch}(v)}
    \bigr)
    $
    with $M_w(p_w)=\lambda_w^\ast\, p_w^{1/d(w)}$, $w\in\mathrm{ch}(v)$.
    Consequently, we have that
    \begin{align*}
      M_v(p)
      =
      \sup_{\prod_{w\in\mathrm{ch}(v)} p_w=p}
      \biggl\{\,
      \sum_{w\in\mathrm{ch}(v)}
      (\lambda_w^\ast)^{-1/\alpha_v}
      p_w^{-1/\{\alpha_v d(w)\}}
      \biggr\}^{-\alpha_v}.
    \end{align*}
    Since $\alpha_v>0$, maximizing the last expression is equivalent to
    minimizing the sum inside the braces.
    For $w\in\mathrm{ch}(v)$, set
    $r_w
      =
      d(w)/d(v)$ and $a_w
      =
      (\lambda_w^\ast)^{-1/\alpha_v}
      p_w^{-1/\{\alpha_v d(w)\}}$.
    Since $r_w>0$ and $\sum_w r_w=1$, the weighted arithmetic--geometric
    mean inequality gives
    \begin{align*}
      \sum_{w\in\mathrm{ch}(v)}a_w
      &=
      d(v)
      \sum_{w\in\mathrm{ch}(v)}
      r_w\frac{a_w}{d(w)}\\
&      \geq
      d(v)
      \prod_{w\in\mathrm{ch}(v)}
      \biggl\{
      \frac{a_w}{d(w)}
      \biggr\}^{r_w}\\
      &=
      d(v)
      p^{-1/\{\alpha_v d(v)\}}
      \prod_{w\in\mathrm{ch}(v)}
      \bigl\{
      d(w)^{\alpha_v}\lambda_w^\ast
      \bigr\}^{-d(w)/\{\alpha_v d(v)\}}.
    \end{align*}
    Equality holds if and only if there exists $\kappa>0$ such that
    $a_w/d(w)
      =
      \kappa$,
     $w\in\mathrm{ch}(v)$,
    or, equivalently,
    \begin{align*}
      p_w
      =
      \biggl\{
      \frac{(\lambda_w^\ast)^{-1/\alpha_v}}
      {\kappa d(w)}
      \biggr\}^{\alpha_v d(w)},
      \quad w\in\mathrm{ch}(v).
    \end{align*}
    The constraint $\prod_{w\in\mathrm{ch}(v)} p_w=p$ then yields
    \begin{align*}
      \kappa^{-\alpha_v}
      =
      p^{1/d(v)}
      \prod_{w\in\mathrm{ch}(v)}
      \bigl\{
      d(w)^{\alpha_v}\lambda_w^\ast
      \bigr\}^{d(w)/d(v)}.
    \end{align*}
    Thus $\kappa$, and consequently $(p_w)_{w\in\mathrm{ch}(v)}$, is
    uniquely determined, and the lower bound is attained. It follows that
    \begin{align*}
      M_v(p)
      =
      d(v)^{-\alpha_v}
      p^{1/d(v)}
      \prod_{w\in\mathrm{ch}(v)}
      \bigl\{d(w)^{\alpha_v}\lambda_w^\ast\bigr\}^{d(w)/d(v)}.
    \end{align*}
    By the definition of $\lambda_v^\ast$
    in~\eqref{eq:recursion:lambda:tree}, this gives
    $M_v(p)=\lambda_v^\ast\,p^{1/d(v)}$.
    In particular, $\lambda^\ast(C_v)=M_v(1)=\lambda_v^\ast$.

By~\eqref{eq:recursion:lambda:tree}, the expression for
$\kappa^{-\alpha_v}$ can be rewritten as
\begin{align*}
  \kappa^{-\alpha_v}
  &=
  p^{1/d(v)}
  \prod_{w\in\mathrm{ch}(v)}
  \{d(w)^{\alpha_v}\lambda_w^\ast\}^{d(w)/d(v)}
  =
  p^{1/d(v)}d(v)^{\alpha_v}\lambda_v^\ast.
\end{align*}
Substituting this into the equality condition above gives
the unique optimal child products in the outer supremum
in~\eqref{eq:Mv:p:decomposed}:
\begin{align}
  p_w^\ast(p)
  =
  \biggl\{
  \frac{\kappa^{-\alpha_v}}
  {d(w)^{\alpha_v}\lambda_w^\ast}
  \biggr\}^{d(w)}
  =
  p^{d(w)/d(v)}
  \biggl\{
  \frac{d(v)^{\alpha_v}\lambda_v^\ast}
  {d(w)^{\alpha_v}\lambda_w^\ast}
  \biggr\}^{d(w)},
  \quad w\in\mathrm{ch}(v).
\label{eq:child:products:optimizer}
\end{align}
   For these optimal products, the induction hypothesis gives a unique
maximizer in each child subtree. Combining these child-subtree maximizers
therefore yields a maximizer in $A_v(p)$.

It remains to prove uniqueness. Let $\bm b_{\mathrm{le}(v)}\in A_v(p)$ be
any maximizer, and set $\tilde p_w=\prod_{i\in\mathrm{le}(w)}b_i$, $w\in\mathrm{ch}(v)$.
By definitions of $M_v$ and $M_w$, we also have
\begin{align*}
  M_v(p)
  =
  S_{\alpha_v}\bigl((\Lambda_w(\bm b_{\mathrm{le}(w)}))_w\bigr)
  \le
  S_{\alpha_v}\bigl((M_w(\tilde p_w))_w\bigr)
  \le
  M_v(p).
\end{align*}
Hence equality holds throughout.
Therefore, the uniqueness of $(p_w^\ast(p))_{w\in\mathrm{ch}(v)}$ yields $\tilde p_w=p_w^\ast(p)$, $w\in\mathrm{ch}(v)$.
Moreover, the strict monotonicity of $S_{\alpha_v}$ implies $\Lambda_w(\bm b_{\mathrm{le}(w)})
  =
  M_w(p_w^\ast(p))$, $w\in\mathrm{ch}(v)$.
By the induction hypothesis, each child-subtree maximizer is unique.
Therefore $\bm b_{\mathrm{le}(v)}$ coincides with the maximizer constructed
above, and $M_v(p)$ is attained by this unique vector in $A_v(p)$.

    This completes the induction and proves part~(i).

  \item If all children of $v$ are leaves, then $d(w)=1$ and
    $\lambda_w^\ast=1$ for every $w\in\mathrm{ch}(v)$. Hence
    \eqref{eq:recursion:lambda:tree} gives
    $
    \lambda_v^\ast=d(v)^{-\alpha_v},
    $
    which agrees with~\eqref{eq:closed:form:lambda:subtree} since $\mathcal I_v\setminus\{v\}=\emptyset$.

    Assume now that~\eqref{eq:closed:form:lambda:subtree} holds for every
    internal child of $v$. By~\eqref{eq:recursion:lambda:tree} and the induction hypothesis, we have that
    \begin{align*}
      \begin{aligned}
        \lambda_v^\ast
        &=
          d(v)^{-\alpha_v}
          \prod_{w\in\mathrm{ch}(v)\cap\mathcal I}
          \bigl\{d(w)^{\alpha_v}\lambda_w^\ast\bigr\}^{d(w)/d(v)}  \\
        &=
          d(v)^{-\alpha_v}
          \prod_{w\in\mathrm{ch}(v)\cap\mathcal I}
          d(w)^{
          (\alpha_v-\alpha_w)d(w)/d(v)
          }
          \prod_{w\in\mathrm{ch}(v)\cap\mathcal I}
         \biggl\{ \prod_{\tilde{w}\in\mathcal I_w\setminus\{w\}}
          d(\tilde{w})^{
          (\alpha_{\mathrm{pa}(\tilde{w})}-\alpha_{\tilde{w}})d(\tilde{w})/d(v)
          }\biggr\}.
      \end{aligned}
    \end{align*}
    Since $w\in\mathrm{ch}(v)\cap\mathcal I$ implies
    $\mathrm{pa}(w)=v$, the first product can be written as
    \begin{align*}
      \prod_{w\in\mathrm{ch}(v)\cap\mathcal I}
      d(w)^{(\alpha_{\mathrm{pa}(w)}-\alpha_w)d(w)/d(v)}.
    \end{align*}
    Moreover, the internal vertices in the subtree rooted at $v$, except for $v$
    itself, decompose disjointly as
    \begin{align*}
        \mathcal I_v\setminus\{v\}
=
(\mathrm{ch}(v)\cap\mathcal I)
\uplus
\biguplus_{w\in\mathrm{ch}(v)\cap\mathcal I}
(\mathcal I_w\setminus\{w\}),
    \end{align*}
    where $\uplus$ denotes disjoint union. 
    Hence
    \begin{align*}
      \lambda_v^\ast
      =
      d(v)^{-\alpha_v}
      \prod_{w\in\mathcal I_v\setminus\{v\}}
      d(w)^{
        (\alpha_{\mathrm{pa}(w)}-\alpha_w)d(w)/d(v)
      },
    \end{align*}
    which proves \eqref{eq:closed:form:lambda:subtree}.

  \item
    Fix $v\in\mathcal I$ and $j\in\mathrm{le}(v)$, and let
    $v_0=v,v_1,\dots,v_m=j$ denote the unique path from $v$ to $j$.
    Let
    \begin{align*}
      P_k^\ast
      =
      \prod_{i\in\mathrm{le}(v_k)}(\bm b_v^\ast)_i,
      \quad k=0,\dots,m,
    \end{align*}
    where $\bm b_v^\ast$ is the unique maximizer for the MTCM of $C_v$,
    whose existence and uniqueness were established in Part~(i) of this proof.
    The uniqueness argument in Part~(i) of this proof
    also shows recursively that the restriction of $\bm b_v^\ast$ to
    $\mathrm{le}(v_k)$ attains $M_{v_k}(P_k^\ast)$ for every
    $k=0,\dots,m$.
    Then $P_0^\ast=1$, $P_m^\ast=(\bm b_v^\ast)_j$,
    $d(v_m)=1$ and $\lambda_{v_m}^\ast=1$. Applying~\eqref{eq:child:products:optimizer} recursively along the path gives
    \begin{align*}
      P_k^\ast
      =
      (P_{k-1}^\ast)^{d(v_k)/d(v_{k-1})}
      \biggl\{
      \frac{
      d(v_{k-1})^{\alpha_{v_{k-1}}}
      \lambda_{v_{k-1}}^\ast
      }{
      d(v_k)^{\alpha_{v_{k-1}}}
      \lambda_{v_k}^\ast
      }
      \biggr\}^{d(v_k)},
      \quad k=1,\dots,m,
    \end{align*}
    or, equivalently,
    \begin{align*}
      (P_k^\ast)^{1/d(v_k)}
      =
      (P_{k-1}^\ast)^{1/d(v_{k-1})}
      \frac{
      d(v_{k-1})^{\alpha_{v_{k-1}}}
      \lambda_{v_{k-1}}^\ast
      }{
      d(v_k)^{\alpha_{v_{k-1}}}
      \lambda_{v_k}^\ast
      }.
    \end{align*}
    Iterating this identity and using $P_0^\ast=1$, we obtain
    \begin{align*}
      \begin{aligned}
        (\bm b_v^\ast)_j
        =
        P_m^\ast
        =
        \prod_{k=1}^m
        \frac{
        d(v_{k-1})^{\alpha_{v_{k-1}}}
        \lambda_{v_{k-1}}^\ast
        }{
        d(v_k)^{\alpha_{v_{k-1}}}
        \lambda_{v_k}^\ast
        }
        =
        \lambda_v^\ast d(v)^{\alpha_v}
        \prod_{k=1}^{m-1}
        d(v_k)^{\alpha_{v_k}-\alpha_{v_{k-1}}}.
      \end{aligned}
    \end{align*}
    Since
    $
    \{v_1,\dots,v_{m-1}\}
    =
    \mathrm{an}_v(j),
    $
    we conclude that
    \begin{align*}
      (\bm b_v^\ast)_j
      =
      \lambda_v^\ast d(v)^{\alpha_v}
      \prod_{w\in\mathrm{an}_v(j)}
      d(w)^{\alpha_w-\alpha_{\mathrm{pa}(w)}}.
    \end{align*}
    Since $v\in\mathcal I$ and $j\in\mathrm{le}(v)$ were arbitrary, this
    proves the result as stated. \qedhere
  \end{enumerate}
\end{proof}
}

\end{document}